\documentclass[10pt,leqno]{amsart}
\usepackage{amsmath}
\usepackage{amscd,amsthm,amssymb,amsfonts}
\usepackage{mathrsfs}
\usepackage{dsfont}
\usepackage{stmaryrd}
\usepackage{euscript}
\usepackage{expdlist}
\usepackage{enumerate}

\input xy
\xyoption{all}
\usepackage[OT2,T1]{fontenc}

\theoremstyle{plain}
\newtheorem{thm}{Theorem}[section]
\newtheorem{thmA}{Theorem}

\newtheorem*{thm*}{Theorem}

\newtheorem{lm}[thm]{Lemma}
\newtheorem{cor}[thm]{Corollary}
\newtheorem*{cor*}{Corollary}
\newtheorem{prop}[thm]{Proposition}
\newtheorem*{conj*}{Conjecture}

\theoremstyle{remark}
\newtheorem*{remark}{Remark}
\newtheorem*{thank}{Acknowledgments}

\theoremstyle{definition}
\newtheorem*{defn*}{Definition}
\newtheorem{Remark}[thm]{Remark}
\newtheorem{defn}[thm]{Definition}

\newcommand{\nc}{\newcommand}

\newcommand{\beq}{\begin{equation}}
\newcommand{\eeq}{\end{equation}}
\newcommand{\bpmx}{\begin{pmatrix}}
\newcommand{\epmx}{\end{pmatrix}}
\newcommand{\bbmx}{\begin{bmatrix}}
\newcommand{\ebmx}{\end{bmatrix}}
\newcommand{\wh}{\widehat}
\newcommand{\wtd}{\widetilde}

\newcommand{\beqcd}[1]{\begin{equation*}\label{#1}\tag{#1}}
\newcommand{\eeqcd}{\end{equation*}}

\numberwithin{equation}{section}
\newenvironment{mylist}{
  \begin{enumerate}{}{%
      \setlength{\itemsep}{5pt} \setlength{\parsep}{0in}
      \setlength{\parskip}{0in} \setlength{\topsep}{0in}
      \setlength{\partopsep}{0in}
      \setlength{\leftmargin}{0.17in}}}{\end{enumerate}}

\def\parref#1{\ref{#1}}
\def\thmref#1{Theorem~\parref{#1}}
\def\propref#1{Prop.~\parref{#1}}
\def\corref#1{Cor.~\parref{#1}}     \def\remref#1{Remark~\parref{#1}}
\def\secref#1{\S\parref{#1}}

\def\lmref#1{Lemma~\parref{#1}}
\def\subsecref#1{\S\parref{#1}}

\def\makeop#1{\expandafter\def\csname#1\endcsname
  {\mathop{\rm #1}\nolimits}\ignorespaces}

\makeop{Hom}   \makeop{End}   \makeop{Aut}   
\makeop{Pic} \makeop{Gal}       \makeop{Div} \makeop{Lie}
\makeop{PGL}   \makeop{Corr} \makeop{PSL} \makeop{sgn} \makeop{Spf}
 \makeop{Tr} \makeop{Nr} \makeop{Fr} \makeop{disc}
\makeop{Proj} \makeop{supp} \makeop{ker}   \makeop{Im} \makeop{dom}
\makeop{coker} \makeop{Stab} \makeop{SO} \makeop{SL} \makeop{SL}
\makeop{Cl}    \makeop{cond} \makeop{Br} \makeop{inv} \makeop{rank}
\makeop{id}    \makeop{Fil} \makeop{Frac}  \makeop{GL} \makeop{SU}
\makeop{Trd}   \makeop{Sp} \makeop{Tr}    \makeop{Trd} \makeop{Res}
\makeop{ind} \makeop{depth} \makeop{Tr} \makeop{st} \makeop{Ad}
\makeop{Int} \makeop{tr}    \makeop{Sym} \makeop{can} \makeop{SO}
\makeop{torsion} \makeop{GSp} \makeop{Tor}\makeop{Ker} \makeop{rec}
\makeop{Ind} \makeop{Coker}
 \makeop{vol} \makeop{Ext} \makeop{gr} \makeop{ad}
 \makeop{Gr}\makeop{corank} \makeop{Ann}
\makeop{Hol} 
\makeop{Fitt} \makeop{Mp} \makeop{CAP}

\def\Spec{\mathrm{Spec}\,}

\DeclareMathAlphabet{\mathpzc}{OT1}{pzc}{m}{it}
\DeclareSymbolFont{cyrletters}{OT2}{wncyr}{m}{n}
\DeclareMathSymbol{\SHA}{\mathalpha}{cyrletters}{"58}

\def\makebb#1{\expandafter\def
  \csname bb#1\endcsname{{\mathbb{#1}}}\ignorespaces}
\def\makebf#1{\expandafter\def\csname bf#1\endcsname{{\bf
      #1}}\ignorespaces}
\def\makegr#1{\expandafter\def
  \csname gr#1\endcsname{{\mathfrak{#1}}}\ignorespaces}
\def\makescr#1{\expandafter\def
  \csname scr#1\endcsname{{\EuScript{#1}}}\ignorespaces}
\def\makecal#1{\expandafter\def\csname cal#1\endcsname{{\mathcal
      #1}}\ignorespaces}

\def\doLetters#1{#1A #1B #1C #1D #1E #1F #1G #1H #1I #1J #1K #1L #1M
                 #1N #1O #1P #1Q #1R #1S #1T #1U #1V #1W #1X #1Y #1Z}
\def\doletters#1{#1a #1b #1c #1d #1e #1f #1g #1h #1i #1j #1k #1l #1m
                 #1n #1o #1p #1q #1r #1s #1t #1u #1v #1w #1x #1y #1z}
\doLetters\makebb   \doLetters\makecal  \doLetters\makebf
\doLetters\makescr
\doletters\makebf   \doLetters\makegr   \doletters\makegr

\def\Gm{{\bbG}_{m}}

\normalsize

\makeop{Ram} \makeop{Rep} \makeop{mass}

\makeop{Bl}
\def\abs#1{\left|#1\right|}
\def\norm#1{\lVert#1\rVert}
\def\Fpbar{\bar{\mathbb F}_p}

\def\Qbarp{\C_p}
\def\Qp{\Q_p}
\def\Qbar{\bar\Q}
\def\Zbar{\bar{\Z}}
\def\Zbarp{\Zbar_p}

\def\Zp{\Z_p}

\def\el{\ell}

\def\rmT{{\mathrm T}}
\def\rmN{{\mathrm N}}

\def\cA{{\mathcal A}}  
\def\cB{\EuScript B}
\def\cD{\mathcal D}
\def\cE{{\mathcal E}}
\def\cF{{\mathcal F}}  

\def\cI{\mathcal I}

\def\cK{{\mathcal K}}  
\def\cM{\mathcal M}
\def\cR{{\mathcal R}}
\def\cO{\mathcal O}

\def\cf{{\mathcal f}}
\def\cW{{\mathcal W}}

\def\cQ{\mathcal Q}

\def\EucA{{\EuScript A}}

\def\EucE{{\EuScript E}}

\def\EucO{{\EuScript O}}

\def\bfc{\mathbf c}

\def\bfM{\mathbf M}

\def\bda{\mathbf a}
\def\bff{\mathbf f}
\def\bdh{\mathbf h}

\def\bfi{\mathbf i}

\def\bdw{\mathbf w}

\def\bdc{\mathbf c}

\def\bfw{\mathbf w}

\def\bfdelta{\boldsymbol{\delta}}
\def\bftheta{\boldsymbol{\theta}}

\def\sL{\mathscr L}

\def\sS{\mathscr S}

\def\bbI{\mathbb I}

\newcommand{\Z}{\mathbf Z}
\newcommand{\Q}{\mathbf Q}
\newcommand{\R}{\mathbf R}
\newcommand{\C}{\mathbf C}
\newcommand{\A}{\mathbf A}    

\def\bbE{{\mathbb E}}

\def\bbmu{\boldsymbol{\mu}}

\def\fraka{{\mathfrak a}}
\def\frakb{{\mathfrak b}}
\def\frakc{{\mathfrak c}}

\def\frakq{\mathfrak q}

\def\frakm{\mathfrak m}

\def\frakl{\mathfrak l}

\def\frakF{{\mathfrak F}}
\def\frakE{\mathfrak E}

\def\frakI{\mathfrak I}

\def\frakC{{\mathfrak C}}

\def\frakX{\mathfrak X}

\def\frakR{\mathfrak R}
\def\frakS{\mathfrak S}

\def\frakN{\mathfrak N}
\def\il{\mathfrak i\frakl}

\def\ulA{\ul{A}}

\def\ulz{\ul{z}}

\def\ollam{\bar{\lam}}

\def\Zhat{\hat{\Z}}
\def\wbar{\bar{w}}

\def\wbar{\bar{w}}
\def\zbar{\bar{z}}

\def\ab{abelian variety }
\def\etale{{\'{e}tale }}

\def\padic{\text{$p$-adic }}

\def\BS{Bruhat-Schwartz }

\def\Teich{Teichm\"{u}ller }
\def\Neron{N\'{e}ron }
\def\Frob{\mathrm{Frob}}

\newcommand{\<}{\langle}   
\renewcommand{\>}{\rangle} 

\def\isoto{\stackrel{\sim}{\to}}

\def\ot{\otimes}

\def\hookto{\hookrightarrow}

\def\ol{\overline}  \nc{\opp}{\mathrm{opp}} \nc{\ul}{\underline}

\newcommand{\pair}[2]{\< #1, #2\>}

\newcommand{\pairing}{\pair{\,}{\,}}

\def\XYmatrix{\xymatrix@M=8pt} 
\def\ncmd{\newcommand}
\ncmd{\xysubset}[1][r]{\ar@<-2.5pt>@{^(-}[#1]\ar@<2.5pt>@{_(-}[#1]}
\ncmd{\XYmatrixc}[1]{\vcenter{\XYmatrix{#1}}}
\ncmd{\xyto}[1][r]{\ar@{->}[#1]}
\ncmd{\xyinj}[1][r]{\ar@{^(->}[#1]}
\ncmd{\xysurj}[1][r]{\ar@{->>}[#1]}
\ncmd{\xyline}[1][r]{\ar@{-}[#1]}
\ncmd{\xydotsto}[1][r]{\ar@{.>}[#1]}
\ncmd{\xydots}[1][r]{\ar@{.}[#1]}
\ncmd{\xyleadsto}[1][r]{\ar@{~>}[#1]}
\ncmd{\xyeq}[1][r]{\ar@{=}[#1]} \ncmd{\xyequal}[1][r]{\ar@{=}[#1]}
\ncmd{\xyequals}[1][r]{\ar@{=}[#1]}
\ncmd{\xymapsto}[1][r]{l\ar@{|->}[#1]}\ncmd{\xyimplies}[1][r]{\ar@{=>}[#1]}
\ncmd{\xyiso}{\ar[r]_-{\sim}}
\def\injxy{\ar@{^(->}}

\newcommand{\MX}[4]{\begin{bmatrix}
{#1}& {#2}\\
{#3}&{#4}\end{bmatrix} }

 \newcommand{\DII}[2]{\begin{bmatrix}{#1}&0
 \\0&{#2}\end{bmatrix}}

\newcommand{\seesaw}[4]{{#1}\ar@{-}[rd]\ar@{-}[d]&{#2}\ar@{-}[d]\\
{#3}\ar@{-}[ru]&{#4}}

\def\ie{i.e. }

\def\cf{\mbox{{\it cf.} }}
\def\loccit{\mbox{{\it loc.cit.} }}

\def\can{{can}}

\def\ENS{\mathfrak E\mathfrak n\mathfrak s}
\def\SCH{\mathfrak S\mathfrak c\mathfrak h}

\def\ch{{\mathbb I}}

\def\uf{\varpi} 
\def\Abs{{|\!\cdot\!|}} 

\def\Sg{{\varSigma}}  

\def\ndivides{\nmid}
\def\ndivide{\nmid}
\def\x{{\times}}

\def\onehalf{{\frac{1}{2}}}
\def\al{\alpha}

\def\kap{\kappa}
\def\om{\omega}
\def\dirlim{\varinjlim}
\def\prolim{\varprojlim}
\def\iso{\simeq}
\def\con{\equiv}
\def\bksl{\backslash}
\newcommand\stt[1]{\left\{#1\right\}}
\def\ep{\epsilon}

\def\l{{\ell}}

\def\lam{\lambda}
\def\pii{\pi i}

\def\sg{\sigma}
\def\vp{\varphi}
\def\disjoint{\bigsqcup}
\def\bigot{\bigotimes}

\def\smid{\,|}

\def\dx{d^\x}

\def\AFf{\A_{\cF,f}}
\def\AKf{\A_{\cK,f}}
\def\AF{\A_\cF}
\def\AK{\A_\cK}
\def\setp{{(p)}}

\def\bbox{{(\Box)}}

\newcommand{\powerseries}[1]{\llbracket{#1}\rrbracket}

\renewcommand\pmod[1]{\,(\mbox{mod }{#1})}
\renewcommand\mod[1]{\,\mbox{mod }{#1}}

\renewcommand\Re{\text{Re}\,}
\def\vphi{\varphi}
\def\Cp{\C_p}

\def\alg{\mathrm{alg}} 
\def\OF{O}
\def\OK{R}
\def\OFv{O_v}
\def\OKv{R_v}

\def\adelef{\A_{\cF,f}}
\def\Section{\phi_{\ads,s,v}}
\def\uf{\varpi}
\def\units{\OFv^\x}

\def\wbar{\bar{w}}
\def\ads{\chi}
\def\nads{\ads^*}
\def\kap{\xi}
\def\NV{{\bf (NV)} }
\def\cmpt{\varsigma}
\def\cmptv{\varsigma_v}
\def\Cinert{\frakI}
\def\Cram{\frakR}
\def\Csplit{\frakF}
\def\ufl{\uf_{\frakl}}

\def\holES{\bbE^h_{\ads}}

\def\adsnu{\ads\nu}
\def\Cl{Cl}
\def\padicf{\EucE}
\def\ENS{SETS}
\def\SCH{SCH}
\def\frakE{\mathfrak E}

\def\antich{\frakX^-_\frakl}
\def\bR{\mathfrak O}
\def\baseR{\cW}
\def\frakN{N}
\def\ab{{\fraka,\frakb}}
\def\Gm{\mathbb G_m}
\def\pads{\widehat{\ads}}
\def\addchar{\psi}

\subjclass[2010]{11F67 11G15}
\thanks{The author is partially supported by National Science Council grant 98-2115-M-002-017-MY2}
\title[Non-vanishing of Hecke $L$-values modulo $p$]{
On the non-vanishing of Hecke $L$-values modulo $p$}
\author[M.-L. Hsieh]{Ming-Lun Hsieh}
\date{August 12, 2012}
\address{ Department of Mathematics~\\National Taiwan University ~ \\
No. 1, Sec. 4, Roosevelt Road, Taipei 10617, Taiwan~
}
\email{mlhsieh@math.ntu.edu.tw}
\begin{document}
\begin{abstract}In this article, we follow Hida's approach to establish an analogue of Washington's theorem on the non-vanishing modulo $p$ of Hecke $L$-values for CM fields with anticyclotomic twists
\end{abstract}
\maketitle
\tableofcontents
\def\ZZbox{\Z_{(\Box)\,}}
\def\cAbox{\cA^{(\Box)}_{K,\bdc}}
\def\cAboxn{\cA^{(\Box)}_{K,\bdc,n}}
\def\Zhatbox{\widehat{\Z}^{(\Box)}}
\def\qchKF{\tau_{\cK/\cF}}
\def\opcpt{K}
\def\sh{Sh}
\def\opn{K^n}
\def\lp{j}
\def\lpp{\eta^{(p)}}
\def\lsgN{\opcpt_1^n}
\def\Om{\boldsymbol{\omega}}
\def\wt{k}
\def\skewhf{\vartheta}
\def\Fv{F}
\def\Kv{E}
\def\nh{{n.h.}}

\def\Ig{I}
\def\wbar{\ol{w}}
\def\Katzd{\theta}
\def\norm#1{\rmN_{\cK/\cF}(#1)}
\def\Beth{\eta}
\def\bR{\EucO}
\def\OKbasis{\bftheta}
\def\Tr{\rmT}

\section*{Introduction}
The purpose of this paper is to study the non-vanishing modulo $p$ property of Hecke $L$-values for CM fields via arithmetic of Eisenstein series. Let $\cF$ be a totally real field of degree $d$ over $\Q$ and $\cK$ be a totally imaginary quadratic extension of $\cF$. Let $\Sg$ be a CM type of $\cK$. Then we can attach the CM period $\Omega_\infty=(\Omega_{\infty,\sg})_{\sg}\in(\C^\x)^\Sg$ to a \Neron differential on an abelian scheme $\EucA_{/\Zbar}$ of CM type $(\cK,\Sg)$. Let $p>2$ be a rational prime and let $\ell\not =p$ be a rational prime and $\frakl$ be a prime of $\cF$ above $\ell$. Let $c$ be the nontrivial element in $\Gal(\cK/\cF)$. We fix an arithmetic Hecke character $\ads$ of $\cK^\x$
with infinity type $k\Sg+\kappa(1-c)$, where $k$ is a positive
integer and $\kappa=\Sigma_{\sg\in\Sg}\kappa_\sg\sg$ with integers
$\kappa_\sg\geq 0$. For a multi-index $\kappa=\sum_{\sg\in\Sg}\kappa_\sg\sg\in\Z[\Sg]$, we write $\Omega_\infty^\kappa=\Omega_{\infty,\sg}^{\kappa_\sg}$ and $a^\kappa=a^{\sum_\sg \kappa_\sg}$ for $a\in\C^\x$.

Let $\cK_{\frakl^n}$ be the ray class field of conductor $\frakl^n$ and let $\cK_{\frakl^\infty}=\cup_n\cK_{\frakl^n}$. Let $\cK^-_{\frakl^\infty}$ be the maximal pro-$\ell$ anticyclotomic extension of $\cK$ in $\cK_{\frakl^\infty}$ and let $\Gamma^-=\Gal(\cK^-_{\frakl^\infty}/\cK)$. Let $\antich$ be the set of finite order characters of $\Gamma^-$.
For every $\nu\in\antich$, we consider the complex number
\[L^{\alg,\frakl}(0,\ads\nu):=\frac{\pi^{\kappa}\Gamma_\Sg(k\Sg+\kappa)L^{(\frakl)}(0,\ads\nu)}{\Omega_\infty^{k\Sg+2\kappa}},\]
where $\Gamma_\Sg(k\Sg+\kappa)=\prod_{\sg\in\Sg}\Gamma(k+\kappa_\sg)$. It is known that $L^{\alg,\frakl}(0,\ads\nu)\in\Zbar_\setp$ if $p$ is uramified in $\cF$ and prime to the conductor of $\ads$ . We are interested in the non-vanishing property of
$L^{\alg,\frakl}(0,\ads\nu)$ modulo $p$ when $\nu$ varies in $\antich$. To be precise, we fix two embeddings
$\iota_\infty:\Qbar\hookto\C$ and $\iota_p:\Qbar\hookto\Qbarp$ once
and for all and let $\frakm$
be the maximal ideal of $\Zbar_{(p)}$ induced by $\iota_p$. We ask if the following non-vanishing modulo $p$ property holds for $(\ads,\frakl)$.
\beqcd{NV}\iota_\infty^{-1}(L^{\alg,\frakl}(0,\ads\nu))\not\con
0\mod{\frakm}\text{ for \emph{almost all} }\nu\in\antich.
\eeqcd
Here \emph{almost all} means "except for finitely many $\nu\in\antich$" if $\dim_{\Q_\ell} F_\frakl=1$ and "Zariski dense subset of $\antich$" if $\dim_{\Q_\ell}F_\frakl>1$ (See \cite[p.737]{Hida:nonvanishingmodp}).

This problem has been studied extensively by Hida for general CM fields in \cite{Hida:nonvanishingmodp} and \cite{Hida:nonvanishingnew} under
the hypothesis that $\Sg$ is $p$-ordinary and by T. Finis in \cite{Finis:nonvanishingell}
for imaginary quadratic fields under a different hypothesis. Let $\qchKF$ be the quadratic character associated to $\cK/\cF$ and $\cD_{\cK/\cF}$ be the different of $\cK/\cF$. Let $\frakC$ be the conductor of $\ads$. The following theorem is proved by Hida in \cite{Hida:nonvanishingnew}.
\begin{thm*} Suppose that $\Sg$ is $p$-ordinary and $p>2$ is unramified in $\cF$. If $(p\frakl,\frakC)=1$ and $\frakC$ is a product of split prime factors over $\cF$, then \NV holds for $(\ads,\frakl)$ unless the following three conditions are
satisfied simultaneously:\begin{itemize}
\item[(M1)] $\cK/\cF$ is unramified everywhere,
\item[(M2)] $\qchKF(\frakc)$ has value $-1$, where $\frakc$ is the polarization ideal of $\EucA_{/\Zbar}$,
\item[(M3)] For all ideal $\fraka$ of $\cF$ prime to $p\frakC$, $\chi
\bfN_{\cF/\Q}(\fraka)\con \qchKF(\fraka)\pmod{\frakm}$.
\end{itemize}
\end{thm*}
We shall say $\chi$ is \emph{residually self-dual} if the condition (M3) holds for $\chi$. By \cite[Lemma 5.2]{Hida:mu_invariant}, the hypotheses (M1-3) is equivalent to the condition (V): $\chi$ is residually self-dual, and the root number associated to $\chi$ is congruent to $-1$ modulo $\frakm$.

We are mainly concerned about the \NV property of \emph{self-dual} characters. Recall that $\ads$ is self-dual if $\ads|_{\AF^\x}=\qchKF\Abs_{\AF}$. Such characters are of its own interest
because an important class of them arises from Hecke characters associated to CM abelian varieties over
totally real fields (\cf \cite[Thm.20.15]{Shimura:ABV-with-CM}). Note that as the conductor of self-dual characters by definition is divisible by ramified primes, these characters in general are not covered in Hida's theorem unless $\cK/\cF$ is unramified. Our main motivation for the \NV property of self-dual characters is the application to Iwasawa main conjecture for CM fields (\cf \cite{Hida:nonvanishingnew} and \cite{Hsieh:ESU21}).
In our subsequent work \cite{Hsieh:ESU21}, this property is used to show the non-vanishing modulo $p$ of the period integral of certain theta functions which is related to Fourier-Jacobi coefficients of Eisenstein series on unitary groups of degree three. When $\cK$ is an imaginary quadratic field and $\frakl$ splits in $\cK$, the problem of the non-vanishing modulo $p$ of Hecke $L$-values associated to self-dual characters has been
solved completely by T. Finis in \cite{Finis:nonvanishingell} through direct study on the period integral of theta functions modulo $p$ (self-dual characters are called \emph{anticyclotomic} in \cite{Finis:nonvanishingell}).

We shall state our main result after preparing some notation. Write $\frakC=\frakC^+\Cinert\Cram$, where $\frakC^+$, $\Cinert$ and $\Cram$ are a product of split, inert and ramified prime factors over $\cF$ respectively. Let $v_p$ be the \padic valuation induced by $\iota_p$. For each $v|\frakC^-$, let $\mu_p(\ads_v)$ be the local invariant defined by
\[\mu_p(\ads_v):=\inf_{x\in\cK_v^\x}v_p(\ads(x)-1).\]
Note that $\mu_p(\ads_v)$ agrees with the one defined in \cite{Finis:nonvanishingell} when $\ads$ is self-dual.
Following Hida, we make the following hypotheses for $(p,\cK,\Sg)$:
\begin{align}\label{unr}\tag{unr}&\text{$p>2$ is unramified in $\cF$};\\
\label{ord}\tag{ord}&\text{$\Sg$ is $p$-ordinary}.
\end{align}
\def\HypNVI{\begin{mylist}
\item[\rm{(L)}] $\mu_p(\ads_v)=0$ for every $v|\frakC^-,$
\item[\rm{(R)}] The global root number $W(\nads)=1$, where $\nads:=\ads\Abs^{-\onehalf}_{\AK},$
\item[\rm{(C)}] $\Cram$ is square-free.
\end{mylist}}
Our main result is as follows.
\begin{thmA}Let $\ads$ be a self-dual Hecke character of $\cK^\x$ such that \HypNVI
In addition to \eqref{unr}, \eqref{ord}, we further assume
\begin{itemize}\item  $(p\frakl,\cD_{\cK/\cF}\frakC)=1$,
\item $\frakl$ splits in $\cK$.
\end{itemize}
Then \NV holds for $(\ads,\frakl)$.
\end{thmA}
Note that as $\ads$ is self-dual, the assumption (R) is equivalent to Hida's condition (V). Indeed, the assumptions (L) and (R) are necessary for the \NV property. The assumption (R) is due to the functional equation of the complex $L$-function $L(s,\ads)$, and the failure of \NV without (L) has been observed by Gillard (\cf\cite[Theorem 1.1]{Finis:nonvanishingell}). We remark that our result in particular can be applied to Hecke characters attached to certain CM elliptic curves over totally real fields. For example, let $E$ be an elliptic curve over $\cF$ with CM by an imaginary quadratic field $\cM$. Let $\cK=\cF\cM$ and let $\ads$ be the Hecke character of $\cK^\x$ such that $L(s,\ads^{-1})=L(E_{/\cF},s)$. Then it is well known that the assumptions (L) and (C) hold if $(\cD_{\cK/\cF},\#(\cO_\cM^\x))=1$ and $p>3$. In general, (C) is expected to be unnecessary. The very reason we impose them is due to the difficulty of the computation of certain Gauss sums $A_\beta(\ads)=A_\beta(\ads_s)|_{s=0}$ defined in \eqref{E:2.N}. We leave the removal of (C) to our forthcoming paper \cite[\S 6]{Hsieh:VMU}.

\bigskip
We also consider the case $\ads$ is not residually self-dual. In particular, this implies the failure of (V). We prove the following result in \corref{C:1.N}, which gives a partial generalization of Hida's theorem.
\begin{thmA}Suppose that \eqref{unr}, \eqref{ord} and $(p\frakl,\cD_{\cK/\cF}\frakC)=1$. Suppose further that the following conditions hold:\begin{itemize}
\item[(L)] $\mu_p(\ads_v)=0$ for every $v|\frakC^-,$
\item[(N)] $\ads$ is not residually self-dual.
\end{itemize}
Then \NV holds for $(\ads,\frakl)$.
\end{thmA}

The proof is based on Hida's ideas in \cite{Hida:nonvanishingmodp}, where Hida provided a general
strategy to study the problem of the non-vanishing of Hecke
$L$-values modulo $p$ via a study on the Fourier coefficients of Eisenstein series. The starting point of Hida is Damerell's formula, which relates a sum of suitable Eisenstein series evaluated at CM points to Hecke $L$-values for CM fields. And then he proves a key result on Zariski density of CM points in Hilbert modular varieties modulo $p$, by which he is able to reduce the problem to non-vanishing of an Eisenstein series modulo $p$ using a variant of Sinnot's argument. The assumption that $\frakC$ is a product of
split primes solely results from the difficulty of the calculation of
Fourier coefficients of Eisenstein series. Following Hida's strategy, we first construct an Eisenstein measure which
interpolates the Hecke $L$-values by the evaluation at CM points. The construction of our Eisenstein measure is from representation theoretic point of view, and Damerell's formula is actually a period integral of Eisenstein series against a non-split torus. Fourier coefficients of our Eisenstein series are decomposed into a product of local Whittaker integrals. Through an explicit calculation of these local integrals, we find that some Fourier coefficient is non-zero modulo $p$ provided that certain epsilon dichotomy holds (See \propref{P:main.N}).

Here is the outline of this article. We fix notation and recall some basic facts about Hilbert modular varieties and CM points in the first three sections. We basically follow the exposition in \cite{Hida:nonvanishingmodp} except that we use an adelic description of CM points.
Readers who are familiar with \cite{Hida:nonvanishingmodp} may begin with \secref{S:ES}, which is the bulk of this paper. In \secref{S:ES}, we give the construction of Eisenstein series and the calculation of some local Whittaker integrals. The formulas of the key integrals $\wtd A_\beta(\ads)$ are summarized in \propref{P:formulaRamified.N} and \propref{P:formula.N}. The explicit calculation of the period integral of our Eisenstein series is carried out in \secref{S:EvCM}. Finally we show some Fourier coefficient of our Eisenstein series is non-zero modulo $p$ in \secref{S:NVES}.

\begin{thank}The author would like to thank Prof. Hida for helpful email correspondence during preparation of this article.
Also the author would like to thank Prof. Sun, Hae-Sang for useful conversation during the stay in Korea Institute of Advanced Study in September 2009. Finally, the author is very grateful to the referee for many valuable suggestions on the improvements of our main results (especially on \lmref{L:2.N} and \corref{C:1.N}) in the previous version of this manuscript.
\end{thank}

\section{Notation and definitions}\label{S:Notation}
\subsection{}Throughout $\cF$ is a totally real field of degree $d$ over $\Q$ and $\cK$ is a totally imaginary quadratic extension of $\cF$. Let $c$
be the complex conjugation, the unique non-trivial
element in $\Gal(\cK/\cF)$. Let $\OF$ (resp. $\OK$) be the ring of integer of $\cF$ (resp. $\cK$). Let $\cD_\cF$ (resp. $D_\cF$) be the different (resp. discriminant) of $\cF/\Q$. Let $\cD_{\cK/\cF}$ be the different of $\cK/\cF$. 
For every fractional ideal $\frakb$ of $\OF$, set $\frakb^*=\frakb^{-1}\cD_\cF^{-1}$. Denote by $\bda=\Hom(\cF,\C)$ the set of archimedean places of $\cF$. Denote by $\bdh$ (resp. $\bdh_\cK$) the set of finite places of $\cF$ (resp. $\cK$). We often write $v$ for a place of $\cF$ and $w$ for the place of $\cK$ above $v$. Denote by $\cF_v$ the completion of $\cF$ at $v$ and by $\uf_v$ a unifomrmizer of $\cF_v$. Let $\cK_v=\cF_v\ot_\cF\cK$.

Fix two rational primes $p\not =\ell$. Let $\frakl$ be a prime of $\cF$ above $\ell$. Let $\Sg$ be a fixed CM type of $\cK$ as in the introduction. We shall identify
$\Sg$ with $\bda$ by the restriction to $\cF$. We assume \eqref{unr} and \eqref{ord} for $(p,\cK,\Sg)$ throughout this article.
Let \[\Sg_p=\stt{w\in \bdh_\cK\mid w|p\text{ and $w$ is induced by $\iota_p\circ\sg$ for $\sg\in\Sg$}}.\]
We recall that $\Sg$ is $p$-ordinary if $\Sg_p\cap\Sg_pc=\emptyset$ and $\Sg_p\cup\Sg_pc=\stt{w\in\bdh_\cK\mid w|p}$.
Note that \eqref{ord} implies that every prime of $\cF$ above $p$ splits in $\cK$.

\subsection{}
If $L$ is a number field, $\A_L$ is the adele of $L$ and $\A_{L,f}$
is the finite part of $\A_L$. The ring of integers of
$L$ is denoted by $\cO_L$. For $a\in\A_L$, we put
\[\il_L(a):=a(\cO_L\ot\Zhat)\cap L.\]
Let $\addchar_\Q$ be the standard additive character
of $\A_\Q/\Q$ such that $\addchar_\Q(x_\infty)=\exp(2\pii x_\infty),\,x_\infty\in\R$. We define $\addchar_L:\A_L/L\to\C^\x$ by
$\addchar_L(x)=\addchar_\Q\circ\Tr_{L/\Q}(x)$. For $\beta\in L$, $\addchar_{L,\beta}(x)=\addchar_L(\beta x)$. If $L=\cF$, we write $\addchar$ for $\addchar_\cF$.

We choose once and for all an embedding
$\iota_\infty:\Qbar\hookto\C$ and an isomorphism
$\iota:\C\iso\Qbarp$, where $\Qbarp$ is the completion of an
algebraic closure of $\Qp$. Let
$\iota_p=\iota\iota_\infty:\Qbar\hookto\Qbarp$ be their composition.
We regard $L$ as a subfield in $\C$ (resp. $\C_p$) via
$\iota_\infty$ (resp. $\iota_p$) and $\Hom(L,\Qbar)=\Hom(L,\C_p)$.

Let $\Zbar$ be the ring of algebraic integers of $\Qbar$ and let $\Zbarp$ be the \padic completion of $\Zbar$ in $\Qbarp$ with the maximal ideal $\frakm_p$. Let $\frakm=\iota_p^{-1}(\frakm_p)$.
\subsection{}
Let $F$ be a local field. Denote by $\Abs_F$ the absolute value of $F$. We often drop the subscript $F$ if it is clear from the context. We fix the choice of our Haar measure $dx$ on $F$. If $F=\R$, $dx$ is the Lebesgue measure on $\R$. If $F=\C$, $dx$ is the twice the Lebesgue measure.
If $F$ is a non-archimedean local field, $dx$ (resp. $\dx x$) is the Haar measure on $F$ (resp. $F^\x$) normalized so that $\vol(\cO_F,dx)=1$ (resp. $\vol(\cO_F^\x,\dx x)=1$).
If $\mu:F^\x\to\C^\x$ is a character of $F^\x$, define
\[a(\mu)=\inf\stt{n\in \Z_{\geq 0}\mid
\mu|_{1+\uf_v^n \OFv}=1}.\]

\section{Hilbert modular varieties and Hilbert modular forms}\label{S:Hilbert}
\subsection{}We follow the exposition in \cite[\S 4.2]{Hida:p-adic-automorphic-forms}. Let $V=\cF e_1\oplus\cF e_2$ be a two dimensional
$\cF$-vector space and $\pairing:V\x V\to \cF$ be the $\cF$-bilinear
alternating pairing defined by $\pair{e_1}{e_2}=1$. Let $\sL=\OF
e_1\oplus \OF^* e_2$ be the standard $\OF$-lattice in $V$. Let $G=\GL_2{}_{/\cF}$. For $g=\MX{a}{b}{c}{d}\in M_2(\cF)$, we define an involution
$g'=\MX{d}{-b}{-c}{a}$. If $g\in G(\cF)=\GL_2(\cF)$, then $g'=g^{-1}\det g$. We
identify vectors in $V$ with row vectors according to the basis $e_1,e_2$, so $G$ has a natural right
action on $V$. Define a left action of $G$ on $V$ by $g*x:=x\cdot g',\,x\in V$.

For each finite place $v$ of $\cF$, we put
\[\opcpt^0_v=\stt{g\in G(\cF_v)\mid g*(\sL\ot_{\OF}\OFv)=\sL\ot_{\OF}\OFv}.\]
Let $\opcpt^0=\prod_{v\in\bdh}\opcpt^0_v$ and $\opcpt^0_p=\prod_{v|p}\opcpt^0_v$.
For a prime-to-$p\ell$ positive integer $\frakN$, we define an open-compact subgroup $U(N)$ of $G(\AFf)$ by
\beq\label{E:opcpt.N}U(\frakN):=\stt{g\in G(\AFf)\mid g\con 1\pmod{\frakN\sL}}.\eeq
 Let $\opcpt$ be an open-compact subgroup of $G(\AFf)$ such that $K_p=\opcpt^0_p$. We assume that $\opcpt\supset U(\frakN)$ for some $\frakN$ as above and that
 $\opcpt$ is sufficiently small so that the following condition holds:
\beqcd{neat}\opcpt\text{ is neat and }\det (\opcpt)\cap
\OF_+^\x\subset (K\cap \OF^\x)^2.\eeqcd
\subsection{Kottwitz models}
We first review Kottwitz models of Hilbert modular varieties.
\begin{defn}[$S$-quadruples]\label{D:6.H} Let $\Box$ be a finite set of
rational primes and let $\baseR_\bbox=\Z_\bbox[\zeta_{N}]$, $\zeta=\exp(\frac{2\pii}{N})$. Define the fibered category $\cA^\bbox_{\opcpt}$ over
$SCH_{/\baseR_\bbox}$ as follows. Let $S$ be a locally noethoerian connected $\baseR_\bbox$-scheme and let $\ol{s}$ be a geometric point of $S$. Objects are abelian varieties with real
multiplication (AVRM) over $S$ of level $\opcpt$, \ie a
$S$-\emph{quadruple} $\ulA=(A,\ollam,\iota,\ol{\eta}^\bbox)_S$ consisting of the
following data:
\begin{enumerate}
\item $A$ is an abelian scheme of dimension $d$ over $S$.
\item $\iota :\OF\hookrightarrow \End_S A\ot_\Z\ZZbox$.
\item $\lam$ is a prime-to-$\Box$ polarization of $A$ over $S$ and
$\ollam$ is the $\OF_{\bbox,+}$-orbit of $\lam$. Namely
\[\ollam=\OF_{\bbox,+}\lam:=\stt{\lam'\in\Hom(A,A^t)\ot_\Z\ZZbox\mid \lam'=\lam\circ a,\,a\in O_{\bbox,+}}.\]
\item $\ol{\eta}^\bbox=\eta^\bbox\opcpt^\bbox$ is a $\pi_1(S,\ol{s})$-invariant $\opcpt^\setp$-orbit of isomorphisms of $\cO_\cK$-modules $\eta^\bbox: \sL\ot_\Z\A_f^\bbox\isoto V^\bbox(A_{\ol{s}}):= H_1(A_{\ol{s}},\A_f^\bbox)$. Here we define $\eta^\bbox g$ for $g\in G(\AFf^\bbox)$ by $\eta^\bbox g(x)=\eta^\bbox(g*x)$.
\end{enumerate}
Furthermore, $(A,\ollam,\iota,\ol{\eta}^\bbox )_S$ satisfies the
following conditions:
\begin{itemize}
\item Let ${}^t$ denote the Rosati involution induced by $\lam$ on
$\End_SA\ot\ZZbox$. Then $\iota(b)^t=\iota(b),\, \forall\, b\in
\OF.$
\item Let $e^\lam$ be the Weil pairing induced by $\lam$. Lifting the isomorphism $\Z/N\Z\iso \Z/N\Z(1)$ induced by $\zeta_N$ to an isomorphism $\zeta:\Zhat\iso\Zhat(1)$, we can regard $e^\lam$ as an $\cF$-alternating form
$e^\lam:V^\bbox(A_{\ol{s}})\times V^\bbox(A_{\ol{s}})\to
\cD^{-1}_\cF\ot_\Z\A_f^\bbox$. Let $e^\eta$ denote the
$\cF$-alternating form on $V^\bbox(A_{\ol{s}})$ induced by
$e^\eta(x,x')=\pair{x\eta}{x'\eta}$. Then
\[e^\lam=u\cdot e^\eta\text{ for some }u\in\AFf^\bbox.\]
\item As
$\OF\ot_\Z\cO_S$-modules, we have an isomorphism $\Lie A\iso \OF\ot_\Z\cO_S$ locally under Zariski topology of $S$.
\end{itemize}
For two $S$-quadruples $\ulA=(A,\ollam,\iota,\ol{\eta}^\bbox )_S$ and $\ul{A'}=(A',\ol{\lam'},\iota',(\ol{\eta'})^\bbox)_S$, we define the morphisms by
\[\Hom_{\cA^\bbox_{\opcpt}}(\ulA,\ul{A'})=\stt{\phi\in \Hom_{\OF}(A,A')\mid
\phi^*\ol{\lam'}=\ollam,\,\phi\circ(\ol{\eta'})^\bbox=\ol{\eta}^\bbox }.\] We
say $\ulA\sim\ul{A'}$ (resp. $\ulA\iso\ul{A'}$) if there exists a
prime-to-$\Box$ isogeny (resp. isomorphism) in
$\Hom_{\cA_\opcpt^\bbox}(\ulA,\ul{A'})$.
\end{defn}
We consider the cases when $\Box=\emptyset$ and $\stt{p}$. When
$\Box=\emptyset$ is the empty set and $\baseR_\bbox=\Q(\zeta_N)$, we define the functor
$\cE_{\opcpt}:\SCH_{/\Q(\zeta_N)}\to\ENS$ by
\[\cE_\opcpt(S)=\stt{\ulA=(A,\ollam,\iota,\ol{\eta})_S\in\cA_\opcpt(S)}/\sim.\] By the theory of Shimura-Deligne, $\cE_{\opcpt}$ is represented by a quasi-projective scheme $\sh_\opcpt$ over $\Q(\zeta_N)$. We define
the functor $\frakE_\opcpt:\SCH_{/\cQ}\to\ENS$ by
\[\frakE_\opcpt(S)=\stt{(A,\ollam,\iota,\ol{\eta})\in\cA^\bbox_\opcpt(S)\mid \eta^\bbox(\sL\ot_\Z\Zhat)=H_1(A_{\ol{s}},\Zhat)}/\iso.\]
By the discussion in \cite[p.136]{Hida:p-adic-automorphic-forms}, we have $\frakE_\opcpt\isoto\cE_\opcpt$ under the hypothesis \eqref{neat}.

When $\Box=\stt{p}$, we write $\baseR$ for $\baseR_\setp$ and define functor
$\cE^\setp_{\opcpt}:\SCH_{/\baseR}\to\ENS$ by
\[\cE^\setp_{\opcpt}(S)=\stt{\ulA=(A,\ollam,\iota,\ol{\eta}^\setp)_S\in\cA_{\opcpt^\setp}^\setp(S)}/\sim.\]
In \cite{Kottwitz:Points-On-Shimura-Varieties}, Kottwitz shows
$\cE^\setp_{\opcpt}$ is representable by a quasi-projective
scheme $\sh^\setp_{\opcpt}$ over $\baseR$ if $\opcpt$ is neat. Similarly we
define the functor $\frakE_K^\setp:\SCH_{/\baseR}\to\ENS$ by
\[\frakE_K^\setp(S)=\stt{(A,\ollam,\iota,\ol{\eta}^\setp)\in\cA^\setp_\opcpt(S)\mid \eta^\setp(\sL\ot_\Z\Zhat^\setp)=H_1(A_{\ol{s}},\Zhat^\setp)}/\iso.\]
It is shown in \cite[\S 4.2.1]{Hida:p-adic-automorphic-forms} that
$\frakE^\setp_K\isoto\cE^\setp_K$.

\subsection{Igusa schemes}\label{17.H}
\begin{defn}[$S$-quintuples] Let $n$ be a positive integer. We define the fibered category $\cA_{K,n}^\setp$ whose
objects are AVRM over an $\baseR$-scheme of level $\opn$, \ie a $S$-quintuple $(\ul{A},j)_S$ consisting of
a $S$-quadruple $\ul{A}=(A,\ollam,\iota,\ol{\eta}^\setp)\in\cA^\setp_{\opcpt^\setp}(S)$ and a monomorphism
\[\lp: \OF^*\ot\bbmu_{p^n}\hookto A[p^n]\] as $\OF$-group
schemes over $S$. We call $\lp$ a level-$p^n$ structure of $A$.
Morphisms are
\[\Hom_{\cA^\setp_{K,n}}((\ulA,j),(\ul{A'},j'))=\stt{\phi\in\Hom_{\cA^\setp_{\opcpt^\setp}}(\ulA,\ul{A'})\mid \phi j=j'}.\]
\end{defn}

Define the functor $\frakI^\setp_{K,n}:\SCH_{/\baseR}\to\ENS$ by
\[\frakI^\setp_{K,n}(S)=\stt{(\ulA,\lp)=(A,\ollam,\iota,\ol{\eta}^\setp,\lp)_S\in
\cA^\setp_{K,n}(S)\mid \lpp(\sL\ot_\Z\Zhat^\setp)=T^\setp(A)}/\iso.
\]
It is known that $\frakI^\setp_{K,n}$ are relatively representable over
$\frakE^\setp_K$ (\cf \cite[Prop. 3.12]{SGA3-2}), so it is
represented by a scheme over $\baseR$, which we denote by $\Ig_{\opcpt,n}$.

For $n\geq n'>0$, the natural morphism
$\pi_{n,n'}:\Ig_{\opcpt,n}\to\Ig_{\opcpt,n'}$ induced by the
inclusion $\OF^*\ot\bbmu_{p^{n'}}\hookto \OF^*\ot\bbmu_{p^n}$ is finite
\etale. The forgetful
morphism $\pi:\Ig_{\opcpt,n}\to \sh^\setp_{\opcpt}$ defined by
$\pi:(\ulA,\lp)\mapsto \ulA$ are \etale for all $n>0$. Hence
$\Ig_{\opcpt,n}$ is smooth over $\Spec\baseR$. The image of $\pi$ is
the pre-image of ordinary abelian schemes in
$\Ig_{\opcpt,n}\ot\Fpbar$.

\subsection{Complex uniformization}\label{S:cpx} We describe the complex points $\sh_\opcpt(\C)$. Put
\[X^+=\stt{\tau=(\tau_\sg)_{\sg\in\bda}\in\C^{\bda}\mid \Im \tau_\sg >0\text{ for all } \sg\in\bda}.\]
Let $\cF_+$ be the set of totally positive elements in $\cF$ and let $G(\cF)^+=\stt{g\in G(\cF)\mid \det g\in\cF_+}$. Define the complex Hilbert modular variety by
\[M(X^+,\opcpt):=G(\cF)^+\bksl X^+\x G(\AFf)/\opcpt.\]
It is well known that $M(X^+,\opcpt)\isoto\sh_\opcpt(\C)$ by the theory
of abelian varieties over $\C$.

For $\tau=(\tau_\sg)_{\sg\in\bda}\in X^+$, we let $p_\tau$ be the
period map $V\ot_\Q\R\isoto \C^{\bda}$ defined by
$p_\tau(ae_1+be_2)=a\tau+b$, $a,b\in \cF\ot_\Q\R=\R^{\bda}$. We can associate a AVRM to $(\tau,g)\in X^+\x G(\AFf)$ as follows.
\begin{itemize}
\item The complex abelian variety $\EucA_g(\tau)=\C^{\bda}/p_\tau(g*\sL)$.
\item  The $\cF_+$-orbit of polarization
$\ol{\pairing}_\can$ on $\EucA_g(\tau)$ is given by the Riemann form $\pairing\circ p_\tau^{-1}$.
\item The $\iota_\C:O\hookto\End \EucA_g(\tau)\ot_\Z\Q$ is induced from the pull back of the natural $\cF$-action on $V$ via $p_\tau$.
\item The level structure $\eta_g:
\sL\ot_\Z\A_f \isoto (g*\sL)\ot_\Z\A_f=H_1(\EucA_g(\tau),\A_f)$ is defined by  $\eta_g(v)= g*v$.\end{itemize}
Let $\ul{\EucA_{g}(\tau)}$ denote the $\C$-quadruple $(\EucA_g(\tau),\ol{\pairing}_\can,\iota,\opcpt\eta_g)$. Then $[(\tau,g)]\mapsto [\ul{\EucA_{g}(\tau)}]$ gives rise to an isomorphism $M(X^+,\opcpt)\isoto \sh_\opcpt(\C)$.

Let $\ulz=\stt{z_\sg}_{\sg\in\bda}$ be the standard complex coordinates of $\C^{\bda}$ and $d\ulz=\stt{dz_\sg}_{\sg\in\bda}$. Then $O$-action on $d\ulz$ is given by
$\iota_\C(\al)^* dz_\sg=\sg(\al)dz_\sg,\,\sg\in\bda=\Hom(\cF,\C)$. Let $z=z_{id}$ be the coordinate corresponding to $\iota_\infty:\cF\hookto\Qbar\hookto\C$.
Then
\beq\label{E:7.N}(\OF\ot_\Z\C) dz=H^0(\EucA_g(\tau),\Omega_{\EucA_g(\tau)/\C}).\eeq
\subsection{Hilbert modular forms}
\subsubsection{}For $\tau\in \C$ and $g=\MX{a}{b}{c}{d}\in\GL_2(\R)$,
we put \beq\label{E:4.N}J(g,\tau)=c\tau+d.\eeq For
$\tau=(\tau_\sg)_{\sg\in\bda}\in X^+$ and $g_\infty=(g_\sg)_{\sg\in\bda}\in
G(\cF\ot_\Q\R)$, we put
\[\ul{J}(g_\infty,\tau)=\prod_{\sg\in\bda}J(g_\sg,\tau_\sg).\]
\begin{defn}
Denote by $\bfM_k(\opcpt,\C)$ the space of holomorphic Hilbert modular form of parallel weight $k$ and level $\opcpt$.
Each $\bff\in\bfM_k(\lsgN,\C)$ is a $\C$-valued function $\bff:X^+\x G(\AFf)\to \C$ such that the function $\bff(-,g_f):X^+\to\C$ is holomorphic for each $g_f\in G(\AFf)$ and
\[\bff(\al(\tau,g_f)u)=\ul{J}(\al,\tau)^{k\Sg}\bff(\tau,g_f)\text{ for all }u\in \lsgN\text{ and }\al\in G(\cF)^+.\]
\end{defn}
\subsubsection{Fourier expansion}
 For every $\bff\in \bfM_k(\lsgN,\C)$, we have the Fourier expansion
\[\bff(\tau,g_f)=\sum_{\beta\in\cF_+\cup\stt{0}}W_\beta(\bff,g_f)e^{2\pii\Tr_{\cF/\Q}(\beta \tau)}.\]We call $W_\beta(\bff,g_f)$ the $\beta$-th Fourier coefficient of $\bff$ at $g_f$.

For a semi-group $L$ in $\cF$, let $L_+=\cF_+\cap L$ and $L_{\geq 0}=L_+\cup \stt{0}$. If $B$ is a ring, we denote by $B\powerseries{L}$ the set of all formal series
\[\sum_{\beta\in L}a_\beta q^\beta,\,a_\beta\in B.\]
Let $a,b\in(\AFf^{(p\frakN)})^\x$ and let $\fraka=\il_\cF(a)$ and
$\frakb=\il_\cF(b)$. The $q$-expansion of $\bff$ at the cusp $(\fraka,\frakb)$ is given by
\beq\label{E:FC0}\bff|_{(\fraka,\frakb)}(q)=\sum_{\beta\in (N^{-1}\fraka\frakb)_{\geq 0}}W_\beta(\bff,\MX{a^{-1}}{0}{0}{b})q^\beta\in \C\powerseries{(N^{-1}\fraka\frakb)_{\geq 0}}.\eeq
 If $B$ is a $\baseR$-algebra in $\C$, we put
\begin{align*}
\bfM_k(\opcpt,B)&=\stt{\bff\in \bfM_k(\opcpt,\C)\mid  \bff|_{(\fraka,\frakb)}(q)\in B\powerseries{(N^{-1}\fraka\frakb)_{\geq 0}}\text{ at all cusps }(\fraka,\frakb)}.\\
\end{align*}
\subsubsection{Tate objects}
Let $\sS$ be a set of $d$-linear $\Q$-independent elements in $\Hom(\cF,\Q)$ such that $l(\cF_+)>0$ for $l\in\sS$. If $L$ is a lattice in $\cF$ and $n$ a positive integer, let
$L_{\sS,n}=\stt{x\in L\mid l(x)>-n\text{ for all }l\in\sS}$ and put $B((L;\sS))=\lim\limits_{n\to\infty} B\powerseries{L_{\sS,n}}$.
To a pair $(\fraka,\frakb)$ of two prime-to-$pN$ fractional ideals , we can attach the Tate AVRM $Tate_{\fraka,\frakb}(q)=\Gm\ot_\Z\fraka^*/q^{\frakb}$ over $\Z((\fraka\frakb;\sS))$ with $O$-action $\iota_\can$. As described in \cite{Katz:p_adic_L-function_CM_fields}, $Tate_{\fraka,\frakb}(q)$ has a canonical $\fraka\frakb^{-1}$-polarization $\lam_\can$ and also carries $\Om_\can$ a canonical $\OF\ot\Z((\fraka\frakb;\sS))$-generator of $\Omega_{Tate_{\fraka,\frakb}}$ induced by the isomorphism $\Lie(Tate_{\fraka,\frakb}(q)_{/\Z((\fraka\frakb;\sS))})=\fraka^*\ot_\Z\Lie(\Gm)\iso\fraka^*\ot\Z((\fraka\frakb;\sS))$.
Let $\sL_{\fraka,\frakb}=\sL\cdot\MX{\frakb}{}{}{\fraka^{-1}}=\frakb e_1\oplus\fraka^*e_2$. Then we have a level $N$-structure $\eta_\can:N^{-1}\sL_\ab/\sL_\ab\isoto Tate_\ab(q)[\frakN]$ over $\Z[\zeta_N]((N^{-1}\fraka\frakb;\sS))$ induced by the fixed primitive $N$-th root of unity $\zeta_N$.
We write $\ul{Tate}_{\fraka,\frakb}$ for the Tate $\Z((\fraka\frakb;\sS))$-quadruple $(Tate_{\fraka,\frakb}(q),\ol{\lam_\can},\iota_\can,\ol{\eta}^\setp_\can)$ at $(\fraka,\frakb)$. In addition, since $\fraka$ is prime to $p$, we let $\eta^0_{p,\can}\colon \OF^*\ot_\Z\bbmu_{p^n}=\fraka^*\ot_\Z\bbmu_{p^n}\hookto Tate_{\fraka,\frakb}(q)$ be the canonical level $p^n$-structure induced by the natural inclusion $\fraka^*\ot_\Z\bbmu_{p^n}\hookto\fraka^*\ot_\Z\Gm$.

\subsubsection{Geometric modular forms}\label{S:GME}We collect here definitions and basic facts of geometric modular forms. For the precise theory, we refer to \cite{Katz:p_adic_L-function_CM_fields} or \cite{Hida:p-adic-automorphic-forms}. Let $T=\Res_{\OF/\Z}\Gm$ and $\kappa\in\Hom(T,\Gm)$. Let $B$ be a $\Z_\setp$-algebra. Consider $[\ulA]=[(A,\ol{\lam},\iota,\ol{\eta}^\setp)]\in\frakE_{\opcpt}(C)$ for a $B$-algebra $C$ with a differential form $\Om$ generating $H^0(A,\Omega_{A/C})$ over $\OF\ot_\Z C$. A geometric modular form $f$ over $B$ of weight $\kappa$ and level $\opcpt$ is a functorial rule of assigning a value $f(\ulA,\Om)\in C$ satisfying the following axioms.
\begin{mylist}
\item[(G1)] $f(\ulA,\Om)=f(\ulA',\Om')\in C$ if $(\ulA,\Om)\iso (\ulA',\Om')$ over $C$,
\item[(G2)]For a $B$-algebra homomorphism $\varphi:C\to C'$, we have
\[f((\ulA,\Om)\ot_C C')=\varphi(f(\ulA,\Om)),\]
\item[(G3)]$f(\ulA,a\Om)=\kappa(a^{-1})f(\ulA,\Om)$ for all $a\in T(C)=(\OF\ot_\Z C)^\x$,
\item[(G4)]$f(\ul{Tate}_\ab,\Om_\can)\in B[\zeta_N]\powerseries{(N^{-1}\fraka\frakb)_{\geq 0}}\text{ at all cusps }(\fraka,\frakb)$.
\end{mylist}
For a positive integer $k$, we regard $k\in\Hom(T,\Gm)$ as the character $t\mapsto \bfN_{\cF/\Q}(t)^k$.
We denote by $\cM_k(\opcpt,B)$ the space of geometric modular forms over $B$ of weight $k$ and level $\opcpt$.

For each $f\in \cM_k(\opcpt,\C)$, we regard $f$ as a holomorphic Hilbert modular form of weight $k$ and level $\opcpt$ by
\[f(\tau,g_f)=f(\EucA_g(\tau),\ol{\pairing}_\can,\iota_\C,\ol{\eta}_g,2\pii dz),\]
where $dz$ is the differential form in \eqref{E:7.N}. By GAGA principle, this gives rise to an isomorphism $\cM_{k}(\opcpt,\C)\isoto\bfM_k(\opcpt,\C)$. As discussed in \cite[\S 1.7]{Katz:p_adic_L-function_CM_fields}, the evaluation $\bff(\ul{Tate}_{\ab},\Om_\can)$ is independent of the auxiliary choice of $\sS$ in the construction of the Tate object. Moreover, we have the following important identity which bridges holomorphic modular forms and geometric modular forms. \[\bff|_{(\fraka,\frakb)}(q)=\bff(\ul{Tate}_{\ab},\Om_\can)\in\C\powerseries{(N^{-1}\fraka\frakb)_{\geq 0}}.\]
By $q$-expansion principle, if $B$ is $\baseR$-algebra in $\C$, then $\cM_{k}(\opcpt,B)=\bfM_k(\opcpt,B)$.
\subsubsection{\padic modular forms}Let $B$ be a \padic ring in $\Cp$. Let $V(\opcpt,B)$ be the space of Katz \padic modular forms over $B$ defined by
\[V(\opcpt,B):=\prolim_m\dirlim_n H^0(\Ig_{\opcpt,n}{}_{/B/p^mB},\cO_{\Ig_{\opcpt,n}}).\]
In other words, Katz \padic modular forms are formal functions on Igusa towers.

For each point $(\ulA,\lp)\in\dirlim_m\prolim_n\Ig_{\opcpt,n}{}_{/B/p^m B}$, the level $p^\infty$-structure $\lp$ induces an isomorphism $\wtd\lp:\Lie(\wh{\bbG}_m)\ot_{\Zp} O^*\isoto \Lie(A)$. Let $dt/t$ be the canonical invariant differential form of $\widehat{\bbG}_m$. Then $\lp_*dt/t:=dt/t\circ \wtd\lp^{-1}$ is a generator of $H^0(A,\Omega_A)$ as a $\OF$-module.
We thus have a natural injection\beq\label{E:padicavatar.N}
\begin{aligned}\cM_k(\opcpt,B)&\hookto V(\opcpt,B)\\
f&\mapsto \wh{f}(\ulA,\lp):=f(\ulA,\lp_*dt/t)
\end{aligned}\eeq
which preserves the $q$-expansions in the sense that $\wh f|_{(\ab)}(q):=\wh f(\ul{Tate}_\ab,\eta^0_{p,\can})=f|_{(\ab)}(q)$. We will call $\wh{f}$ the \padic avatar of $f$.

\subsection{Hecke action}Let $h\in G(\AFf^\setp)$ and let ${}_hK:=hK h^{-1}$. We define a morphism
$\smid h:\cE^\setp_{{}_hK}\isoto\cE^\setp_{K}$ by
\[\ulA=(A,\ollam,\iota,\ol{\eta}^\setp )\mapsto \ulA\smid h=(A,\ollam,\iota,h\ol{\eta}^\setp ).\]
Then $\smid h$ induces an $\baseR$-isomorphism $\sh^\setp_{\opcpt}\isoto
\sh^\setp_{\opcpt_h}$, and $\smid h$ thus acts on spaces of modular forms. In particular,  for $F\in V(\opcpt,\baseR)$, we define $F|h\in V({}_h\opcpt,\baseR)$ by
\[F|h(\ulA)=F(\ulA\smid h).\]
Let $\opcpt_0(\frakl):=\stt{g\in \opcpt\mid e_2g\in \OF^*e_2\pmod{\frakl\sL}}$. Define the $U_\frakl$-operator on $V(\opcpt_0(\frakl),\baseR)$ by
\[F|U_\frakl=\sum_{u\in\OF^*/\frakl\OF^*}F|\MX{\ufl}{u}{0}{1}.\]

Using the description of complex points
of $\sh^\setp_{K}(\C)$ in \secref{S:cpx}, it is not difficult to verify by definition that for $(\tau,g)\in X^+\x
G(\AFf)$ two pairs $(\ul{\EucA_{g}(\tau)}\smid h,\Om)$ and $(\ul{\EucA_{gh}(\tau)},\Om)$ of $\C$-quadruples and invariant differential forms are $\Z_\setp$-isogenous,
so we have the isomorphism:
\beq\label{E:HeckeAct}\begin{aligned}
\bfM_k(\opcpt,\C)&\isoto \bfM_k({}_h\opcpt,\C)\\
\bff&\mapsto\bff|h(\tau,g)=\bff(\tau,gh).\end{aligned}\eeq 

\section{CM points}\label{S:CMpoint} \subsection{}\label{S:CM1} In this section, we give an adelic description of CM points in Hilbert modular varieties.
Fix a prime-to-$p$ integral ideal $\frakC$ of $\OK$ such that $(p\frakl,\frakC\cD_{\cK/\cF})=1$. Write $\frakC=\frakC^+\frakC^-$, where $\frakC^-=\Cinert\Cram$,
$\Cinert$ (resp. $\Cram$) is a product of inert (resp. ramified) primes in $\cK/\cF$
and $\frakC^+=\Csplit\Csplit_c$ is a product of split primes in
$\cK/\cF$ such that $(\Csplit,\Csplit_c)=1$ and $\Csplit\subset\Csplit_c^c$.
Recall that we have assumed \eqref{unr} and \eqref{ord} in the introduction. Let $\Sg$ be a $p$-ordinary CM type of $\cK$ and identify
$\Sg$ with $\bda$ by the restriction to $\cF$. We choose $\skewhf\in\cK$ such that
\begin{itemize}
\item[(d1)] $\skewhf^c=-\skewhf$ and
$\Im\sg(\skewhf)>0$ for all $\sg\in\Sg$,
\item[(d2)] $\frakc(\OK):=\cD_{\cF}^{-1}(2\skewhf\cD_{\cK/\cF}^{-1})$ is prime to $pD_{\cK/\cF}\frakl\frakC\frakC^c$.
\end{itemize}
Let
$\skewhf^\Sg:=(\sg(\skewhf))_{\sg\in\Sg}\in X^+$. Let $D=-\skewhf^2\in \cF_+$ and define
$\rho:\cK\hookto M_2(\cF)$ by
\[\rho(a\skewhf+b)=\MX{b}{-D a}{a}{b}.\]Consider the isomorphism $q_\skewhf:\cK\isoto \cF^2=V$ defined by $q_\skewhf(a\skewhf+b)=ae_1+be_2$.
It is clear that $(0,1)\rho(\al)=q_\skewhf(\al)$
and $q_\skewhf(x \al)=q_\skewhf(x)\rho(\al)$ for $\al,x\in \cK$. Let $\C(\Sg)$ be the $\cK$-module whose underlying space is
$\C^\Sg$ with the $\cK$-action given $\al(x_\sg)=(\sg(\al)x_\sg)$.
Then we have a canonical isomorphism $\cK\ot_\Q\R=\C(\Sg)$, and $p_\skewhf:=q_\skewhf^{-1}:V\ot_\Q\R\isoto \cK\ot_\Q\R=\C(\Sg)$ is the period map associated to $\skewhf^\Sg$.
\subsection{A good level structure}\label{S:CM2}
\subsubsection{}\label{S:different}For each $v|p\Csplit\Csplit^c$, we decompose $v=w\wbar$ into two places $w$ and $\wbar$ of $\cK$ with $w|\Csplit\Sg_p$. Here $w|\Csplit\Sg_p$ means $w|\Csplit$ or $w\in\Sg_p$. Let $e_{w}$ (resp.
$e_{\wbar}$) be the idempotent associated to $w$ (resp. $\wbar$). Then $\stt{e_w,e_{\wbar}}$ gives an $\OFv$-basis of $\OKv$. Let $\skewhf_w\in \cF_v$ such that $\skewhf=-\skewhf_we_{\wbar}+\skewhf_w e_w$.

For inert or ramified place $v$ and $w$ the place of $\cK$ above $v$, we fix a $\OFv$-basis $\stt{1,\bftheta_v}$ such that $\bftheta_v$ is a uniformizer if $v$ is ramified and $\ol{\bftheta}=-\bftheta$ if $v\ndivides 2$. Let $\delta_v:=\bftheta_v-\ol{\bftheta_v}$ be a fixed generator of the relative different $\cD_{\cK_w/\cF_v}$.

Fix a finite idele $d_\cF=(d_{\cF_v})\in \adelef$ such that $\il_\cF(d_\cF)=\cD_\cF$. By (d2), we may choose $d_{\cF_v}=2\skewhf\delta_v^{-1}$ if $v|D_{\cK/\cF}\Cinert$ (resp. $d_{\cF_v}=-2\skewhf_w$ if $w|\Csplit\Sg_p$).
\subsubsection{}
We shall choose a basis $\stt{e_{1,v},e_{2,v}}$ of $R\ot_\OF \OF_v$ for each finite
place $v\not =\frakl$ of $\cF$.
If $v\ndivides p\frakl\frakC\frakC^c$, we choose
$\stt{e_{1,v},e_{2,v}}$ in $\OK\ot\OF_v$ such that
$R\ot_{\OF}\OF_v=\OF_v e_{1,v}\oplus \OF_v^* e_{2,v}$. It is clear
that $\stt{e_{1,v},e_{2,v}}$ can be taken to be $\stt{\skewhf,1}$
except for finitely many $v$. If $v|p\Csplit\Csplit^c$, let $\stt{e_{1,v},e_{2,v}}=\stt{e_{\wbar},d_{\cF_v}\cdot e_{w}}$ with $w|\Csplit\Sg_p$. If $v$ is inert or ramified, let
$\stt{e_{1,v},e_{2,v}}=\stt{\bftheta_v,d_{\cF_v}\cdot 1}$.
For every integer $n\geq 0$, we let $R_n=R+\frakl^n R$, and let $\stt{e^{(n)}_{1,\frakl},e^{(n)}_{2,\frakl}}:=\stt{-1,-d_{\cF_\frakl}\uf_\frakl^n\bftheta_\frakl}$ be a basis of
$R_n\ot_\OF\OF_\frakl$.

For $v\in\bdh$, let $\cmptv$ (resp. $\cmpt^{(n)}_\frakl$) be
the element in $\GL_2(\cF_v)$ such that
$e_i\cmptv'=q_\skewhf(e_{i,v})$ (resp.
$e_i(\cmpt^{(n)}_{\frakl})'=q_\skewhf(e^{(n)}_{i,\frakl})$). For
$v=\sg\in \bda$, let $\cmptv=\MX{\Im\sg(\skewhf)}{0}{0}{1}$.
Define $\cmpt=\prod_{v\not =\frakl}\cmptv\in\GL_2(\AF^{(\frakl)})$ and $\cmpt^{(n)}=\cmpt\x
\cmpt^{(n)}_{\frakl}\in\GL_2(\AF)$. Let $\cmpt_f$ and $\cmpt^{(n)}_f)$ be the finite components of $\cmpt$ and $\cmpt^{(n)}$ respectively. By the definition of
$\cmpt^{(n)}$, we have
\[\cmpt^{(n)}_f*(\sL\ot_\Z\Zhat)=(\sL\ot_\Z\Zhat)\cdot(\cmpt^{(n)}_f)'=q_\skewhf(\OK_n\ot_\Z\Zhat).\]
The matrix representation of $\cmptv$ according to the basis $\stt{e_1,e_2}$ for $v|p\frakl D_{\cK/\cF}\frakC\frakC^c$ is given as follows:
\beq\begin{aligned}\label{E:cm.N}
\cmptv&=\MX{d_{\cF_v}}{-2^{-1}t_v}{0}{d_{\cF_v}^{-1}},\,t_v=\bftheta_v+\ol{\bftheta_v}\text{ if $v|D_{\cK/\cF}\Cinert$},\\
\cmptv&= \MX{\frac{d_{\cF_v}}{2}}{-\onehalf}{\frac{d_{\cF_v}}{-2\skewhf_w}}{\frac{-1}{2\skewhf_w}}=\MX{-\skewhf_w}{-\onehalf}{1}{\frac{-1}{2\skewhf_w}}
\text{ if }v|p\Csplit\Csplit^c\text{ and }w|\Csplit\Sg_p,\\
\cmpt^{(n)}_{\frakl}&=\MX{-d_{\cF_\frakl}\uf_\frakl^nb_\frakl }{1}{ d_{\cF_\frakl}\uf_\frakl^na_\frakl }{0}\quad(\bftheta_\frakl=a_\frakl\skewhf+b_\frakl,\,a_\frakl\in\cF^\x_\frakl,b_\frakl\in\cF_\frakl).
\end{aligned}\eeq

\subsection{}
\def\CMring{W}
For every $a\in \AKf^\x$, we let
\[\ul{A}_n(a)_{/\C}:=\ul{\EucA_{\rho(a)\cmpt^{(n)}}(\skewhf^\Sg)}=(\EucA_{\rho(g)\cmpt^{(n)}}(\skewhf^\Sg),\ol{\pairing}_\can,\iota_\can,\ol{\eta(a)})\in\sh_\opcpt(\C)\] be the $\C$-quadruple associated to $(\skewhf^\Sg,\rho(a)\cmpt_f^{(n)})$ as in \secref{S:cpx}.
Then $\ul{A}_n(a)_{/\C}$ is an
abelian variety with CM by $\cK$. Let $\CMring$ be the \padic completion of the maximal unramified extension of $\Zp$ in $\Cp$. By the general theory of CM abelian varieties, the $\C$-quadruple $\ul{A}_n(a)_{/\C}$ descends to a $\CMring$-quadruple $\ulA_n(a)$. Moreover, since $\cK$ is $p$-ordinary, $\ulA_n(a)\ot_\CMring\Fpbar$ is an ordinary abelian variety, hence the level $p^\infty$-structure $\eta(a)_p$ over $\C$ descends to a level $p^\infty$-structure over $\CMring$.
Thus we obtain a map $x_n:\AKf^\x\to\prolim_m\Ig_{\opcpt,m}(\CMring)\subset \Ig_{\opcpt,\infty}(\CMring)$, which factors through $C_\cK:=\AKf^\x/\cK^\x$ the idele class group of $\cK$.
The collection of points $\Cl^\infty:=\disjoint_{n=1}^\infty x_n(C_\cK)$ in $\Ig_{\opcpt,\infty}(\CMring)$ is called \emph{CM points} in Hilbert modular varieties.

\subsection{Polarization ideal} The alternating pairing $\pairing:\cK\x\cK:\to\cF$
defined by $\pair{x}{y}=(c(x)y-xc(y))/2\skewhf$ induces an isomorphism
$\OK\wedge_{\OF} \OK=\frakc(\OK)^{-1}\cD_\cF^{-1}$ for the fractional
ideal $\frakc(\OK)=\cD_\cF^{-1}(2\skewhf \cD_{\cK/\cF}^{-1})$. Then $\frakc(\OK)$ is the polarization of CM points $x_0(1)$. From the equation
\[\cD_\cF^{-1}\det(\cmpt_f)=\wedge^2\sL \cmpt_f'=\wedge^2\OK=\frakc(\OK)^{-1}\cD_\cF^{-1},\]
we find that $\frakc(\OK)=(\det(\cmpt_f))^{-1}$. Moreover, for $a\in\AK^\x$, the polarization ideal of $x_0(a)$ is $\frakc(\fraka):=\frakc(\OK)\rmN_{\cK/\cF}(\fraka)^{-1}$, $\fraka=\il_\cK(a)$.

\subsection{Measures associated to $U_\frakl$-eigenforms}\label{S:Measure}
\subsubsection{}We briefly recall Hida's construction of the measure associated to an $U_\frakl$-eigenform in \cite[\S 3]{Hida:nonvanishingmodp}. Define the compact subgroup $U_n= (\C_1)^\Sg\x(R_n\ot\Zhat)^\x$ in $\AK^\x=(\C^\x)^\Sg\x\AKf^\x$, where $\C_1$ is the unit circle in $\C^\x$. Let $\Cl_n=\cK^\x\AF^\x\bksl \AK^\x/U_n$ and let $[\cdot]_n:\AK^\x\to\Cl_n$ be the quotient map. Let $\Cl_\infty=\prolim_n\Cl_n$. For $a\in\AK^\x$, we let $[a]:=\prolim_n[a]_n\in\Cl_\infty$ be the holomorphic image in $\Cl_\infty$. Henceforth, every $\nu\in\antich$ will be regarded implicitly as a \padic character of $\Cl_\infty$ by \emph{geometrically} normalized reciprocity law.

Let $\padicf\in V(\opcpt_0(\frakl),\EucO)$ for some finite extension $\EucO$ of $\Zp$ and let $\wh\ads$ be the \padic avatar of $\ads$. Assuming the following:
\begin{mylist}\item[(i)] $\padicf$ is a $U_\frakl$-eigenform with the eigenvalue $a_\frakl(\padicf)\in\Zbarp^\x$;
\item[(ii)] $\padicf(x_n(ta))=\pads^{-1}(a)\padicf(x_n(t)),\, a\in U_n\cdot \AF^\x$,\end{mylist}
Hida in \cite[(3.9)]{Hida:nonvanishingmodp} associates a $\Zbarp$-valued measure $\vp_{\padicf}$ on $\Cl_\infty$ to the $U_\frakl$-eigenform $\padicf$ such that for a function $\phi:\Cl_n\to\Zbarp$, we have
\beq\label{E:measure.N}\int_{\Cl_\infty}\phi
d\vp_{\padicf}:=a_\frakl(\padicf)^{-n}\cdot \sum_{[t]_n\in\Cl_n}\padicf(x_n(t))\pads(t)\phi([t]_n).\eeq

\subsubsection{}
Let $\Delta$ be the torsion subgroup of $\Cl_\infty$. Let $\Cl^\alg$ be the subgroup of $\Cl_\infty$ generated by $[a]$ for $a\in(\AK^{(\frakl)})^\x$ and $\Delta^\alg=Cl^\alg\cap \Delta$.
We choose a set of representatives $\cB=\stt{b}$ of
$\Delta/\Delta^\alg$ in $\Delta$ and a set of representatives $\cR=\stt{r}$ of
$\Delta^\alg$ in $(\AKf^{(p\frakl)})^\x$. Thus
$\Delta=\cB[\cR]=\stt{b[r]}_{b\in\cB,r\in\cR}$. For $a\in(\AKf^{(p\frakl)})^\x$, we define
\[\padicf|[a]:=\padicf|\rho_{\cmpt}(a),\,\rho_{\cmpt}(a):=\cmpt^{-1}\rho(a)\cmpt\in G(\AFf^{(p\frakl)}).\] By definition, $\padicf|[a](x_n(t))=\padicf(x_n(ta))$. Following Hida (\cf\cite[(4.4) p.25]{Hida:nonvanishingnew}), we put
\beq\label{E:average.N}\padicf^\cR=\sum_{r\in\cR}\pads(r)\padicf|[r].\eeq

In \cite{Hida:nonvanishingmodp}, Hida reduces the non-vanishing of $L$-values to the non-vanishing of Eisenstein series by proving the following theorem.
\begin{thm}[Theorem 3.2 and Theorem 3.3 \cite{Hida:nonvanishingmodp}]\label{T:1.N}
Suppose the following
conditions in addition to $(\mathrm{unr})$ and $(\mathrm{ord})$:
\begin{description}[\breaklabel\setlabelstyle{\itshape}]\item [$(\mathrm{H})$] Write the
order of the Sylow $\ell$-subgroup of $\bbF[\ads]^\x$
as $\ell^{r(\ads)}$. Then there exists a strict ideal class
$\frakc\in\Cl_\cF$ such that $\frakc=\frakc(\fraka)$ for some
$R$-ideal $\fraka$ and for every $u\in O$ prime to $\frakl$, we
can find $\beta\con u\mod{\frakl^{r(\ads)}}$ with
$\bfa_\beta(\padicf^\cR,\frakc)\not \con 0\pmod{\frakm_p}$,
\end{description}
where $\bfa_\beta(\padicf^\cR,\frakc)$ is the $\beta$-th Fourier coefficient of $\padicf^\cR$ at the cusp $(O,\frakc^{-1})$.
Then
\[\int_{\Cl_\infty}\nu d\padicf\not \con 0\pmod{\frakm_p}\text{ for almost all $\nu\in\antich$}.\]
\end{thm}
\begin{remark}As pointed by the referee, if $\frakl$ has degree one over $\Q$, the above theorem is Theorem 3.2 \cite{Hida:nonvanishingmodp}. In general, the theorem holds under the assumption (h) in Theorem 3.3 \loccit, which is slightly weaker than (H) (See the discussion \cite[p.778]{Hida:nonvanishingmodp}).
\end{remark}
\section{Construction of the Eisenstein series}\label{S:ES}
\subsection{}\label{S:41}
Let $\ads$
be a Hecke character of $\cK^\x$ with infinity type
$k\Sg+\kappa(1-c)$, where $k\geq 1$ is an integer and $\kappa=\sum \kappa_\sg\sg\in \Z[\Sg]$, $\kappa_\sg\geq 0$. Let $\frakc(\ads)$ be the conductor of $\ads$. We assume that $\frakC=\frakc(\ads)\frakS$, where $\frakS$ is only divisible by primes split in $\cK/\cF$ and $(\frakc(\ads)\frakl,\frakS)=1$. Put\[\nads=\ads\Abs^{-\onehalf}_{\AK}\text{ and }\ads_+=\ads|_{\AF^\x}.\]
Let $\opcpt^0_\infty:=\prod_{v\in\bda}\SO(2,\R)$ be a maximal compact subgroup of $G(\cF\ot_\Q\R)$. For $s\in\C$, we let $I(s,\ads_+)$ denote the space consisting of smooth and $\opcpt^0_\infty$-finite functions $\phi:G(\AF)\to \C$ such that \[\phi(\MX{a}{b}{0}{d}g)=\ads_+^{-1}(d)\abs{\frac{a}{d}}_{\AF}^s\phi(g).\]
Conventionally, functions in $I(s,\ads_+)$ are called \emph{sections}. Let $B$ be the upper triangular subgroup of $G$. The adelic Eisenstein series associated
to a section $\phi\in I(s,\ads_+)$ is defined by
\[E_\A(g,\phi)=\sum_{\gamma\in B(\cF)\bksl G(\cF)}\phi(\gamma g).\]
The series $E_\A(g,\phi)$ is absolutely convergent for $\Re s\gg 0$.
\subsection{Fourier coefficients of Eisenstein series} Put
$\bdw=\MX{0}{-1}{1}{0}$. Let $v$ be a place of $\cF$ and let $I_v(s,\ads_+)$ be the local constitute of $I(s,\ads_+)$ at $v$. For $\phi_v\in I_v(s,\ads_+)$ and $\beta\in \cF_v$, we recall that the $\beta$-th local Whittaker integral $W_\beta(\phi_v,g_v)$ is defined by \begin{align*}W_\beta(\phi_v,g_v)=&
\int_{\cF_v}\phi_v(\bdw\MX{1}{x_v}{0}{1}g_v)\addchar(-\beta x_v)dx_v,
\intertext{and the intertwining operator $M_\bdw$ is defined by}
M_\bdw\phi_v(g_v)=&\int_{\cF_v}\phi_v(\bdw \MX{1}{x_v}{0}{1}g_v)dx_v.\end{align*}
By definition, $M_\bdw\phi_v(g_v)$ is the $0$-th local Whittaker integral. It is well known that local Whittaker integrals converge absolutely for $\Re s\gg 0$, and have meromorphic continuation to all $s\in\C$.

If $\phi=\ot_v\phi_v$ is a decomposable section, then it is
well known that $E_\A(g,\phi)$ has the following Fourier expansion:
\beq\label{E:WF.N}\begin{aligned}&E_\A(g,\phi)=\phi(g)+M_\bdw\phi(g)+\sum_{\beta\in\cF}W_\beta(E_\A,g),\text{ where }\\
M_\bdw\phi(g)=&\frac{1}{\sqrt{\abs{D_\cF}_\R}}\cdot \prod_vM_\bdw\phi_v(g_v)\,;\,
W_\beta(E_\A,g)=\frac{1}{\sqrt{\abs{D_\cF}_\R}}\cdot\prod_vW_\beta(\phi_v,g_v).\end{aligned}\eeq
The sum $\phi(g)+M_\bdw\phi(g)$ is called the \emph{constant term} of $E_\A(g,\phi)$.

\subsection{The choice of local sections and Fourier coefficients}
In this subsection, we will choose for each place $v$ a good local
section $\Section$ in $I_v(s,\ads_+)$ and calculate its local $\beta$-th Fourier coefficient for $\beta\in \cF_v^\x$.

\label{S:Local_ES}\subsubsection{}\label{S:ES_notation}
We first introduce some notation and definitions. Let $S^\circ=\stt{v\in\bdh\mid v\ndivide\frakl\frakC\frakC^c D_{\cK/\cF}}$. Let $v$ be a place of $\cF$. Let $L/\cF_v$ be a finite extension and let $d_L$ be a generator of the absolute different $\cD_L$ of $L$. Let $\addchar_L:=\addchar\circ\Tr_{L/\cF_v}$. Given a character $\mu:L^\x\to\C$, we recall that the epsilon factor $\ep(s,\mu,\addchar_L)$ in \cite{Tate:Number_theoretic_bk} is defined by
\begin{align*}\ep(s,\mu,\addchar_L)&=\abs{c}_L^s\int_{c^{-1}\cO_L^\x}\mu^{-1}(x)\addchar_L(x)d_Lx\quad(c=d_L\uf_L^{a(\mu)}).\end{align*}
Here $d_Lx$ is the Haar measure on $L$ self-dual with respect to $\addchar_L$. If $\vp$ is a \BS function on $L$, the zeta integral $Z(s,\mu,\vp)$ is given by
\[Z(s,\mu,\vp)=\int_{L}\vp(x)\mu(x)\abs{x}_L^s\dx x\quad(s\in\C).\]

The local root number $W(\mu)$ is defined by \[W(\mu):=\ep(\onehalf,\mu,\addchar_L)\] (\cf \cite[p.281 (3.8)]{Murase-Sugano:Local_theory_primitive_theta}). It is well known that $\abs{W(\mu)}_\C=1$ if $\mu$ is unitary.

To simplify the notation, we let $\Fv=\cF_v$ (resp.
$\Kv=\cK\ot_{\cF}\cF_v$) and let $d_F=d_{\cF_v}$ be the fixed generator of the absolute different $\cD_F$ in \subsecref{S:different}.
Write $\ads$ (resp. $\ads_+$, $\nads$) for $\ads_v$ (resp. $\ads_{+,v}$, $\nads_v$).
If $v\in\bdh$, we let $\OFv=\cO_{\Fv}$ (resp.
$\OKv=\OK\ot_{\OF}\OFv$) and let $\uf=\uf_v$ be a uniformizer of $F$. For a set $Y$, denote by $\bbI_Y$ the characteristic function of $Y$.
\subsubsection{$v$ is archimedean}\label{S:infinite} Let $v=\sg\in \Sg$ and $\Fv=\R$. For $g\in G(\Fv)=\GL_2(\R)$, we put
\[\bfdelta(g)=\abs{\det (g)}\cdot \abs{J(g,i)\ol{J(g,i)}}^{-1}.\]
Define the section $\phi^h_{k,s,\sg}\in I_v(s,\chi_+)$ of weight $k$ by
\beq\phi^h_{k,s,\sg}(g):=J(g,i)^{-k}\bfdelta(g)^s.\eeq
The intertwining operator $M_\bdw\phi_{k,s,\sg}$ is given by
\beq\label{E:M2.H}M_\bdw\phi^h_{k,s,\sg}(g)=i^{k}(2\pi)\frac{\Gamma(k+2s-1)}{\Gamma(k+s)\Gamma(s)}\cdot \ol{J(g,i)}^k\det(g)^{-k}\bfdelta(g)^{1-s}.\eeq

For $(x,y)\in \R\x\R_+$ and $\beta\in\R^\x$, it is well known that
\beq\label{E:FC1.N}
W_\beta(\phi^h_{k,s,\sg},\MX{y}{x}{0}{1})|_{s=0}=\frac{(2\pii)^k}{\Gamma(k)}\sg(\beta)^{k-1}\exp(2\pi
i\sg(\beta)(x+iy))\cdot\ch_{\R_+}(\sg(\beta)).\eeq Define the section $\phi^{\nh}_{k,\kappa_\sg,s,\sg}\in I(s,\chi_+)$ of weight $k+2\kappa_\sg$ by
\beq\label{E:5.N}\phi^{\nh}_{k,\kappa_\sg,s,\sg}(g):=J(g,i)^{-k-\kappa_\sg}\ol{J(g,i)}^{\kappa_\sg}\bfdelta(g)^s.\eeq
Let $V_+$ be the \emph{weight raising} differential operator in \cite[p.165]{Jacquet_Langlands:GLtwo} given by
\[V_+=\MX{1}{0}{0}{-1}\ot 1+\MX{0}{1}{1}{0}\ot i\in\Lie(\GL_2(\R))\ot_\R\C.\]
Denote by $V_+^{\kappa\sg}$ the operator $(V^+)^{\kappa_\sg}$ acting on $I_v(s,\ads_+)$. By \cite[Lemma 5.6 (iii)]{Jacquet_Langlands:GLtwo}, we have
\beq\label{E:6.N}V_{+}^{\kappa_\sg\sg}\phi^h_{k,s,\sg}=\frac{2^{\kappa_\sg}\Gamma(k+\kappa_\sg+2s)}{\Gamma(k+2s)}\phi^\nh_{k,\kappa_\sg,s,\sg}.\eeq

\subsubsection{$v\in S^\circ$}\label{S:spherical} In this case, $\ads$ is
unramified. Define $\Section(g)$ to be the spherical Godement section in $I_v(s,\ads_+)$. To be precise, put
\begin{align*}\Section(g)=&f_{\Phi_v}(g):=\abs{\det g}^s\int_{\Fv^\x}\Phi_v((0,t)g)\ads_+(t)\abs{t}^{2s}\dx t,\text{ where }\\
&\Phi_v=\ch_{\OFv\oplus \OFv^*}.
\end{align*}
It is well known that the local Whittaker integral is
\beq\label{E:FC2.N}W_\beta(\Section,\MX{1}{}{}{\bfc_v^{-1}})|_{s=0}=\ads_+(\bfc_v)\cdot\frac{1-\nads(\uf)^{v(\beta\bfc_v)+1}}{1-\nads(\uf)}\cdot\abs{\cD_\cF}^{-1}\cdot\bbI_{\OFv}(\beta\bfc_v),
\eeq
and the intertwining operator is given by
\beq\label{E:M3.H}M_\bdw\Section(\MX{1}{}{}{\bfc_v^{-1}})=L_v(2s-1,\ads_+)\abs{\bfc_v}^{1-s}.\eeq

\subsubsection{$v|\Csplit\Csplit^c$}\label{S:Csplit} If $v|\Csplit\Csplit^c$ is split in $\cK$, write $v=w\wbar$ with $w|\Csplit$ and $\ads_v=(\ads_w,\ads_{\wbar})$. Then $a(\ads_{w})\geq
a(\ads_{\wbar})$. We shall define our local section at $v$ to be the Godement section associated to certain \BS functions. We first introduce some \BS functions. For a character $\mu: F^\x\to\C^\x$, we define
\[\vphi_\mu(x)=\ch_{\OFv^\x}(x)\mu(x)\quad(x\in F).\]
Define $\vphi_w=\vphi_{\ads_{w}}$ and
\[\vphi_{\wbar}=\begin{cases}\vphi_{\ads_{\wbar}^{-1}}&\text{ if $\ads_{\wbar}$ is ramified},\\
\ch_{O_v}&\text{ if $\ads_{\wbar}$ is unramified}.
\end{cases}\]
Let $\Phi_v(x,y)=\vphi_{\wbar}(x)\wh{\vphi}_w (y)$, where $\wh{\vphi}_w$ is the Fourier transform of $\vphi_w$ defined by
\[\wh\vphi_w(y)=\int_{\Fv}\vphi_w(x)\addchar(yx)dx.\]
Define $\Section\in I_v(s,\ads_+)$ by \beq\Section(g)=f_{\Phi_v}(g):=\abs{\det
g}^s\int_{F^\x}\Phi_v((0,t)g)\ads_+(t)\abs{t}^{2s}\dx t.\eeq

A straightforward calculation shows that the local Whittaker integral
is \beq\label{E:FC3.N}
\begin{aligned}W_\beta(\Section,1)&=\int_{\Fv^\x}\vphi_{\wbar}(x)\wh{\vphi}_w(-\beta
x^{-1})\cdot\ads_+(x)\abs{x}^{2s-1}\dx x \\
&=\int_{\Fv^\x}\vphi_{\wbar}(x)\vphi_w(\beta
x^{-1})\cdot\abs{\cD_F}^{-1}\cdot\ads_+(x)\abs{x}^{2s-1}\dx x\\
&=\ads_+(\beta)\vphi_{\wbar}(\beta)\abs{\beta}^{2s-1}\cdot\abs{\cD_\cF}^{-1},\end{aligned} \eeq and the intertwining operator is given by
\beq\label{E:M1.H} M_\bdw\Section(1)=0.\eeq

\subsubsection{$v=\frakl$}\label{S:frakl} Let $\phi_{\ads,v,s}\in
I_v(s,\ads_+)$ be the unique $N(\OFv)$-invariant section
supported in the big cell $B(\Fv)\bdw N(\OFv^*)$ and $\phi_{\ads,v,s}(\bdw)=1$. One checks easily that $\Section|U_\frakl$ given by
\[\Section|U_\frakl(g)=\sum_{u\in\OFv^*/\frakl\OFv^*}\Section(g\MX{\uf}{u}{0}{1})\] is also supported in the big cell and is invariant by $N(\OFv^*)$. In particular,
$\Section$ is an $U_\frakl$-eigenform, and the eigenvalue is $\chi_+^{-1}(\uf_\frakl)$.

The local $\beta$-th Whittaker integral is given by
\beq\label{E:FC4.N} \begin{aligned}W_\beta(\phi_{\ads,v,s},\MX{1}{}{}{\bfc^{-1}})|_{s=0}&=\int_{F}\phi_{\ads,v,s}(\bdw\MX{1}{x}{}{1}\MX{1}{}{}{\bfc^{-1}})\addchar(-\beta x)dx|_{s=0}\\
&=\abs{\bfc}\ch_{\OFv}(\beta\bfc),\end{aligned}\eeq
and the intertwining operator is given by $M_\bdw\Section(\DII{1}{\bfc^{-1}})|_{s=0}=\abs{\bfc}$.
\def\UF{\delta}
\def\twobeta{\beta}
\subsubsection{$v|D_{\cK/\cF}\frakC^-$}\label{S:CramCinert} In this case, $\Kv$ is a field and $G(F)=B(F)\rho(\Kv^\x)$. Let $w$ be the place of $\Kv$ above $v$ and let $\uf_E$ be a uniformizer of $E$. Let $\Cram'$ be the product of ramified primes where $\ads$ is unramified, \ie $\Cram'=\prod_{\frakq|\cD_{\cK/\cF},\frakq\ndivide \frakC^-}\frakq$. Let $\Section$ be the unique function on
$G(F)$ such that
\beq\label{E:Dinert.N}\Section(\MX{a}{b}{0}{d}\rho(z)\cmptv)=L(s,\ads_v)\cdot \ads_+^{-1}(d)\abs{\frac{a}{d}}^s\cdot\ads^{-1}(k),\,b\in
B(F),\,z\in \Kv^\x,\eeq
where $L(s,\ads_v)$ is the local Euler factor of $\ads_v$ defined by
\[L(s,\ads_v)=\begin{cases}1&\text{ if }v|\Cram\Cinert,\\
\frac{1}{1-\ads_v(\uf_E)\abs{\uf_E}_E^s}&\text{ if }v|\Cram'.\end{cases}\]
Then $\Section$ defines a smooth section in
$I_v(s,\ads_+)$.

To calculate the local $\beta$-th Whittaker integral of $\Section$, we recall that in \subsecref{S:different}, we have fixed $\UF=\UF_v=2\skewhf d_F^{-1}$ a generator of $\cD_{\cK/\cF}$ and an $\OFv$-basis $\stt{1,\OKbasis}=\stt{1,\bftheta_v}$ of $\OKv$ so that $\UF=2\OKbasis$ if $v\ndivides 2$ and $\UF=\OKbasis-\ol{\OKbasis}$ if $v|2$. In addition, $\OKbasis$ is a uniformizer of $\OKv$ if $v$ is ramified. Let $t=t_v=2\OKbasis-\UF=\OKbasis+\ol{\OKbasis}\in \OFv$. Let $\psi^\circ(x):=\psi(-d_F^{-1}x)$ and $\ads_s=\ads\Abs_E^s$ for $s\in\C$. For $\Re s\gg 0$, we have
\begin{align*}&\frac{1}{L(s,\ads_v)}\cdot W_{\beta}(\Section,1)\\
=& \int_\Fv\Section(\bfw\MX{1}{x+2^{-1}td_F^{-1}}{0}{1}\DII{d_F^{-1}}{d_F}\cmptv)\addchar^\circ(d_F\beta
x)dx\\
=&\addchar^\circ(-2^{-1}t\beta)\ads_s(d_F)\abs{d_F^{-2}}\cdot
\int_F\Section(\MX{\frac{1}{x^2+D}}{\frac{-x}{x^2+D}}{0}{1}\MX{x}{-D}{1}{x}\cmptv)\addchar^\circ(-d_F^{-1}\beta x)dx\\
=&\addchar^\circ(-2^{-1}t\beta)\ads_s(d_F)\abs{d_F^{-2}}\cdot
\int_F\ads_s^{-1}(x+\skewhf)\addchar^\circ(d_F^{-1}\beta
x)dx\\
=&\addchar^\circ(-2^{-1}t\beta)\abs{d_F^{-1}}\cdot
\int_F\ads_s^{-1}(x+2^{-1}\delta)\addchar^\circ(\beta
x)dx.\end{align*}
We put
\beq\label{E:2.N}A_\beta(\ads_s):=\int_F\ads_s^{-1}(x+2^{-1}\UF)\addchar^\circ(\beta x)dx.\eeq
Making change of variable, we find that
\beq\label{E:1.N}\begin{aligned}A_\beta(\ads_s)=&\addchar^\circ(2^{-1}t\beta)\cdot \wtd A_\beta(\ads_s), \text{ where}\\
\wtd A_\beta(\ads_s)=&\int_F\ads^{-1}_s(x+\OKbasis)\addchar^\circ(\twobeta x)dx.\end{aligned}\eeq
In particular, the intertwining operator $M_\bdw\Section(1)=\abs{d_F^{-1}}\cdot L(s,\ads_v)\wtd A_0(\ads_s)$.
We investigate the analytic behavior of $\wtd A_\beta(\ads_s)$. For $\beta\in F$ and $M\geq v(\frakC^-)$, we have\[\wtd A_\beta(\ads_s)=\int_{\uf^{-M}\OFv}\ads_s^{-1}(x+\OKbasis)\addchar^\circ(\twobeta x)dx+\sum_{j>M}\nads\Abs^s(\uf^j)\int_{\OFv^\x}\ads^{-1}(x)\addchar^\circ(\twobeta \uf^{-j}x)dx \]
for $\Re s$ sufficiently large. This shows that $\wtd A_\beta(\ads_s)$ has meromorphic continuation to all $s\in\C$, and $\wtd A_\beta(\ads_s)$ is holomorphic at $s=0$ except when $\beta=0$, $\ads_v$ is unramified and $\nads(\uf)=1$. In particular, when $k\geq 2$, $\wtd A_0(\ads_s)$ and the intertwining operator $M_\bdw\Section(1)$ are finite at $s=0$.

Now we have
\beq\label{E:FC9.N}
W_\beta(\Section,1)|_{s=0}=\abs{d_F^{-1}}\cdot L(s,\ads_v)\wtd A_\beta(\ads_s)|_{s=0}.
\eeq
In what follows, we let $\beta\in F^\x$ and let $\wtd A_\beta(\ads):=\wtd A_\beta(\ads_s)|_{s=0}$ (resp. $A_\beta(\ads):=A_\beta(\ads_s)|_{s=0}$). Then $A_\beta(\ads)=\addchar^\circ(2^{-1}t\beta)\cdot \wtd A_\beta(\ads)$. On the other hand, by our choice of $\OKbasis$ it is clear that $\abs{x+\OKbasis}_E\geq \abs{\OKbasis}_E$ for all $x\in \uf^{-M}\OFv$. Let $M_{\frakC,\beta}=\max\stt{v(\frakC^-),v(\frakC^-)+v(\twobeta)}$. Then for $M\geq M_{\frakC,\beta}$, we have
\beq\label{E:ABchi.N}\begin{aligned}\wtd A_\beta(\ads)&=\int_{\uf^{-M}\OFv}\ads^{-1}(x+\OKbasis)\addchar^\circ(\twobeta x)dx\\
&=\abs{\uf}^{1+M}\sum_{x\in \OFv/(\uf^{1+2M})}\ads^{-1}(\uf^{-M}x+\OKbasis)\addchar^\circ(\twobeta\uf^{-M}x).\end{aligned}
\eeq

The following lemma shows the $p$-integrality of these local Whittaker integrals $W_{\beta}(\Section,1)|_{s=0}$.
\begin{lm}\label{L:ramified_case.N}Suppose that $v|D_{\cK/\cF}\frakC^-$. There exists a finite extension $\EucO$ of $\cO_{\cF,\setp}$ such that for all $\beta\in F^\x$
\[W_{\beta}(\Section,1)|_{s=0}\in\EucO.\] In addition, if $v|\frakC^-$, then \beq\label{E:10.N}\begin{aligned}W_{\beta}(\Section,1)|_{s=0}=\addchar^\circ(-2^{-1}t\beta)\abs{d_F^{-1}}\cdot A_\beta(\ads)
=\abs{d_F^{-1}}\cdot \wtd A_\beta(\ads).\end{aligned}\eeq
If $v|\Cram'$ and $v(\twobeta)=-1$, then
\beq\label{E:FC8.N}
W_{\beta}(\Section,1)|_{s=0}=\abs{d_F^{-1}}\ads^{-1}(\OKbasis)\abs{\uf}.
\eeq
\end{lm}
\begin{proof} First note that it is not difficult to deduce from \eqref{E:ABchi.N} that $\wtd A_\beta(\ads)=0$ if $v(\twobeta)<-1-M_{\frakC,\beta}$, and that $\wtd A_\beta(\ads)$ belongs to the $\cO_{\cF,\setp}$-algebra generated by the values of $\ads$ and $\addchar^\circ(\uf^{-2v(\frakC^-)-1})$ for all $\beta\in F^\x$. Then the assertions are clear if $v|\frakC^-$. Suppose that $v|\Cram'$. Then $v$ is ramified in $E$ and $\ads$ is unramified at $v$.
For $s\in\C$, let $\al_s:=\ads_s(\uf)\abs{\uf}^{-1}$. For $\Re s\gg 0$, we have
\begin{align*} \wtd A_\beta(\ads_s)&=\int_F\ads^{-1}_s(x+\OKbasis)\addchar^\circ(\twobeta x)\\
&=\int_{\uf\OFv}\ads_s^{-1}(x+\OKbasis)\addchar^\circ(\twobeta x)dx+\sum_{j=0}^{v(\twobeta)+1}\abs{\uf^{-j}}\ads_s(\uf^j)\int_{\OFv^\x}\addchar^\circ(\twobeta \uf^{-j}x)dx\\
&=\ads_s^{-1}(\OKbasis)\int_{\uf\OFv}\addchar^\circ(\twobeta x)+\sum_{j=0}^{v(\twobeta)+1}\al_s^j\cdot (\int_{\OFv}\addchar^\circ(\twobeta\uf^{-j} x)dx -\int_{\uf \OFv}\addchar^\circ(\twobeta\uf^{-j} x)dx).
\end{align*}
If $v(\twobeta)\geq 0$, we find that
\begin{align*}\wtd A_\beta(\ads_s)=&\ads_s^{-1}(\OKbasis)\abs{\uf}+\frac{1-\al_s^{v(\twobeta)+1}}{1-\al_s}\cdot (1-\abs{\uf})-\abs{\uf}\al_s^{v(\twobeta)+1}\\
=&(1-\ads_s(\OKbasis))(1+\abs{\uf}\ads_s^{-1}(\OKbasis))\cdot\frac{1-\al_s^{v(\twobeta)+2}}{1-\al_s}.
\end{align*}
In addition, we have $\wtd A_\beta(\ads_s)=0$ if $v(\twobeta)<-1$ and \[\wtd A_\beta(\ads_s)=(1-\ads_s(\OKbasis))\cdot \ads_s^{-1}(\OKbasis)\abs{\uf}\text{ if }v(\twobeta)=-1.\] In any case, the assertions in the case $v|\Cram'$ follow immediately from \eqref{E:FC9.N} and the formulas of $\wtd A_\beta(\ads_s)$.
\end{proof}
\subsubsection{Calculation of $\wtd A_\beta(\ads)$}\label{S:compuation_A.S}
We give an explicit calculation of the local integral $\wtd A(\ads)$ under the assumption $w(\frakC^-)=1$. Introduce an auxiliary integral $\cI(\beta)$ for $\beta\in F$ defined by
\[\cI(\beta):=\int_{\OFv}\ads^{-1}(x+\OKbasis)\addchar^\circ(\twobeta x)dx.\]
The explicit formulas of $\wtd A_\beta(\ads)$ are deduced from the following two lemmas.
\begin{lm}\label{L:preformula.N}Suppose $w(\frakC^-)\leq 1$. \begin{mylist}
\item If $v(\twobeta)\geq 0$ and $\ads|_{\OFv^\x}\not =1$, then
\begin{align*}\wtd A_\beta(\ads)
=&\cI(0)+\ads^{-1}\Abs(-d_F^{-1}\beta)\cdot \ep(1,\ads_+\Abs^{-1},\addchar).
\end{align*}
\item If $v(\twobeta)\geq 0$ and $\ads|_{\OFv^\x}=1$, then \[
\wtd A_\beta(\ads)=\cI(0)+\sum_{j=1}^{v(\twobeta)}\nads(\uf^{j})\cdot(1-\abs{\uf})-\nads(\uf^{v(\twobeta)+1})\cdot\abs{\uf}\quad(\nads=\ads\Abs_E^{-\onehalf}).\]
\item If $v(\twobeta)<0$, then $\wtd A_\beta(\ads)=\cI(\beta)$.
\end{mylist}
\end{lm}
\begin{proof}It is clear that under our assumption on the conductor of $\ads$, we have
\begin{align*}\wtd A_\beta(\ads)&=\int_F\ads^{-1}(x+\OKbasis)\addchar^\circ(\twobeta x)\\
&=\cI(\beta)+\sum_{j=1}^{v(\twobeta)+1}\abs{\uf}^{-j}
\int_{\units}\ads^{-1}(\frac{x}{\uf^j})\addchar^\circ(\frac{\twobeta x}{\uf^j})dx.\end{align*}
Then the lemma follows immediately.
\end{proof}

\begin{lm}\label{L:formula_wtdcI.N}\noindent
\begin{mylist}
\item If $v(\twobeta)<-1$, then $\cI(\beta)=0$.
\item If $v$ is ramified, then $\cI(0)=\nads(\OKbasis^{-1})\abs{\uf}^\onehalf$.
\item If $v$ is inert and $\ads|_{\OFv^\x}=1$, then $\cI(0)=-\abs{\uf}$.
\end{mylist}
\end{lm}
\begin{proof}Note that if $v$ is ramified, then $\ads|_{\OFv^\x}\not =1$. (1) and (2) follows from the assumption on the conductor of $\ads$ and a simple calculation. (3) follows from the following equation:
\begin{align*}0=\int_{\OKv^\x}\ads^{-1}(r)dr&=\int_{\OFv} da\int_{\units}db\ads^{-1}(a+b\OKbasis)+\int_{\units}da\int_{\uf\OFv}db\ads^{-1}(a+b\OKbasis)\\
&=(1-\abs{\uf})\int_{\OFv}\ads^{-1}(a+\OKbasis)da+\abs{\uf}(1-\abs{\uf}).\qedhere
\end{align*}
\end{proof}
We summarize the formulas of $\wtd A_\beta(\ads)$ in the following two propositions.
\begin{prop}\label{P:formulaRamified.N}Suppose that $v|\frakC^-$ is ramified such that $w(\frakC^-)=1$. Then the formula of $\wtd A_\beta(\ads)$ is given as follows.\begin{mylist}
\item If $v(\twobeta)\geq -1$, then
\[\wtd A_\beta(\ads)
=\nads(\OKbasis^{-1})\abs{\uf}^\onehalf+\nads(-\twobeta d_F^{-1}) \ep(1,\ads_+\Abs^{-1},\addchar).\]
\item If $v(\twobeta)<-1$, then $\wtd A_\beta(\ads)=0$.
\item If $v\ndivides 2$ and $v(\beta)\geq -1$, then
\[\wtd A_\beta(\ads)=(\nads(-2\UF^{-1} d_F)+\nads(2^{-1}\beta)W(\nads))\cdot\ads(-d_F^{-1})\abs{\uf}^\onehalf.\]
\end{mylist}
\end{prop}
\begin{proof}In this case, $\OKbasis$ is a uniformizer of $\OKv$. It is straightforward to verify that if $v(\twobeta)=-1$, then
\[\cI(\beta)=\nads(\beta d_F^{-1}) \ep(1,\ads_+\Abs^{-1},\addchar)+\nads(\OKbasis^{-1})\abs{\uf}^\onehalf.\]
Thus (1) and (2) follows from \lmref{L:preformula.N} and \lmref{L:formula_wtdcI.N}.

Suppose $v\ndivides 2$. Then $\delta=2\bftheta$, and (3) follows from (1) and the identity (\cite[Prop.8]{Rohrlich:root_number})
\[W(\nads)=\nads(2)W(\nads|_{F})=\ads(2)\abs{\uf}^{-\onehalf} \ep(1,\ads_+\Abs^{-1},\addchar).\qedhere\]
\end{proof}

\begin{prop}\label{P:formula.N}Suppose that $v|\frakC^-$ is inert such that $w(\frakC^-)=1$.
Then the formula of $\wtd A_\beta(\ads)$ is given as follows.
\begin{mylist}
\item If $v(\twobeta)=-1$, then \[\wtd A_\beta(\ads)=\abs{\uf}\cdot\sum_{a\in\OFv/(\uf)}\ads^{-1}(a+\OKbasis)\addchar^\circ(\twobeta a).\]
\item If $v(\twobeta)<-1$, then $\wtd A_\beta(\ads)=0$.
\item If $v(\twobeta)\geq 0$ and $\ads|_{\OFv^\x}=1$, then
\[
\wtd A_\beta(\ads)=-\abs{\uf}+\sum_{j=1}^{v(\twobeta)}\nads(\uf^{j})\cdot (1-\abs{\uf})-\nads(\uf^{v(\twobeta)+1})\abs{\uf}.\]
\item If $v(\twobeta)\geq 0$ and $\ads|_{\OFv^\x}\not =1$, then
\[\wtd A_\beta(\ads)=\cI(0)+\nads(-\twobeta d_F^{-1}) \ep(1,\ads_+\Abs^{-1},\addchar).\]
\end{mylist}
\end{prop}
\begin{proof} In this case, both $\UF$ and $\OKbasis$ are units of $R_v^\x$. It follows from the definition of $\cI(\beta)$ that
if $v(\twobeta)=-1$, then
\[\cI(\beta)=\abs{\uf}\cdot\sum_{a\in\OFv/(\uf)}\ads^{-1}(a+\OKbasis)\addchar^\circ(\twobeta a).\]
The proposition follows from \lmref{L:preformula.N} and \lmref{L:formula_wtdcI.N} immediately.
\end{proof}

\subsection{Normalization of Eisenstein series}\label{S:normalization}
\begin{defn}For $\bullet=h$ or $\nh$, we put
\[\phi^\bullet_{\ads,s}=\bigot_{\sg\in
\bda}\phi_{k,s,\sg}^\bullet\bigot_{v\in\bdh}\Section.\]
Define the adelic Eisenstein series $E^\bullet_{\ads}$ by
\beq\label{E:ES1.N}
E^\bullet_{\ads}(g)=E_\A(g,\phi_{\ads,s}^\bullet)|_{s=0}, \, \bullet=h,\nh
\eeq
We define the holomorphic (resp. nearly holomorphic) Eisenstein series $\holES$ (resp. $\bbE^\nh_\ads$) by \begin{align*}
\bbE^h_{\ads}(\tau,g_f)&:=\frac{\Gamma_\Sg(k\Sg)}{\sqrt{\abs{D_\cF}_\R}(2\pi
i)^{k\Sg}}\cdot  E^h_{\ads}\left(g_\infty,
g_f\right)\cdot \ul{J}(g_\infty,\bfi)^{k\Sg},\\
\bbE^\nh_{\ads}(\tau,g_f)&:=\frac{\Gamma_\Sg(k\Sg)}{\sqrt{\abs{D_\cF}_\R}(2\pi
i)^{k\Sg}}\cdot E^\nh_{\ads}\left(g_\infty,
g_f\right)\cdot \ul{J}(g_\infty,\bfi)^{k\Sg+2\kappa}(\det g_\infty)^{-\kappa},\\
&\quad ((\tau,g_f)\in X^+\x G(\AFf),\,g_\infty\in G(\cF\ot_\Q\R),\,g_\infty\bfi=\tau,\,\bfi=(i)_{\sg\in\Sg}).
\end{align*}
\end{defn}
Choose $\frakN=\bfN_{\cK/\Q}(\frakC \cD_{\cK/\cF})^m$ for a sufficiently large integer $m$ so that $\Section$ are invariant by $U(\frakN)$ for every $v|\frakN$, and put $\opcpt:=U(\frakN)$. Then the section $\phi_{\ads,s}$ is invariant by $\opcpt_0(\frakl)$.
\begin{prop}\label{P:1.N}Let $\bfc=(\bfc_v)\in\AFf^\x$ be a finite idele such that $\bdc_v=1$ at $v|p\frakl\frakC\frakC^c D_{\cK/\cF}$. Let $\frakc=\il_\cF(\bdc)$. Suppose either of the following conditions holds:
\begin{enumerate}
\item $k>2$,
\item $\Csplit\not =\OF$,
\item $\ads_+=\qchKF\Abs_{\AF}$ and $\frakl$ is split in $\cK$.
\end{enumerate}
Then $\holES\in\bfM_k(\opcpt_0(\frakl),\C)$. The $q$-expansion of $\holES$ at the cusp $(O,\frakc^{-1})$ has no constant term and is given by
\[\holES|_{(\OF,\frakc^{-1})}=\sum_{\beta\in (N^{-1}\frakc^{-1})_+}\bfa_\beta(\holES,\frakc) \cdot q^\beta\in \bR\powerseries{(\frakN^{-1}\frakc^{-1})_+}\] for some finite extension $\bR$ of $\cO_{\cF,\setp}$ in $\Zbar$, where
\begin{align*}\bfa_\beta(\holES,\frakc)=\frac{1}{\abs{D_\cF}_\R}\cdot \bfN_{\cF/\Q}(\beta)^{k-1}\cdot \prod_{v\in\bdh}W_\beta(\Section,\MX{1}{}{}{\bfc_v^{-1}})|_{s=0}.\end{align*}
\end{prop}
\begin{proof}
Let $\phi^{(\infty)}_{\ads,s}=\ot_{v\in\bdh}\Section$. First we claim that for $g_\infty=(g_\sg)_{\sg\in\Sg}\in G(\cF\ot_\Q\R)$,
\beq\label{E:42.N} M_\bdw\phi_{\ads,s}((g_\infty,\MX{1}{}{}{\bfc^{-1}}))|_{s=0}=\prod_{v=\sg\in\Sg}M_\bdw\phi_{k,s,\sg}^h(g_\sg)\cdot M_\bdw\phi^{(\infty)}_{\ads,s}(\MX{1}{}{}{\bfc^{-1}})|_{s=0}=0.\eeq
Indeed, if $k>2$, then \eqref{E:42.N} follows from \eqref{E:M2.H} together with the finiteness of $L^{(p\frakC)}(-1,\ads_+)$, while if there exists some place $v|\Csplit$, then we find \eqref{E:42.N} in view of \eqref{E:M1.H} immediately. If $\ads_+=\qchKF\Abs_{\AF}$ and $\frakl$ splits in $\cK$, then $k=1$ and $\chi$ is ramified at $v|D_{\cK/\cF}$. By \eqref{E:M3.H}, we find that
\begin{align*}M_\bdw\phi^{(\infty)}_{\ads,s}(\MX{1}{}{}{\bfc^{-1}})|_{s=0}&=\prod_{v|\frakC\frakl}M_\bdw\Section(1)|_{s=0}\cdot L^{(\frakl\frakC)}(0,\qchKF)\abs{\bfc}_{\AF},\end{align*}
so \eqref{E:42.N} follows from the fact that
\[L^{(\frakl\frakC)}(0,\qchKF)=L^{(\frakC)}(0,\qchKF)(1-\qchKF(\frakl))=0.\]
This proves the claim. In addition, we have $\phi^{(\infty)}_{\ads,s}(\MX{1}{}{}{\bfc^{-1}})=0$ since $\phi_{\ads,\frakl,s}$ is supported in the big cell. Therefore, we derive the $q$-expansion of $\holES$ from \eqref{E:WF.N} and \eqref{E:FC1.N}.

We study the $p$-integrality of $\bfa_\beta(\holES,\frakc)$. It is well known that $\ads$ takes value in a number field. By the inspection of formulas in \eqref{E:FC2.N}, \eqref{E:FC3.N}, \eqref{E:FC4.N}, \lmref{L:ramified_case.N}, \eqref{E:ABchi.N}, we find that \[\prod_{v\ndivide p\infty}W_\beta(\Section,\MX{1}{}{}{\bfc_v^{-1}})|_{s=0}\in\bR\] for some finite extension $\bR$ of $\cO_{\cF,\setp}$. Since $\ads$ is unramified at $p$, we have $\bfa_\beta(\holES,\frakc)=0$ if $\beta\not\in \OF\ot\Zp$ by \eqref{E:FC2.N}. Note that $v(\nads(\uf_v))=1-k$ and $v'(\nads(\uf_v))=0$ if $v|p$ and $v'\not =v$. Thus if $\beta\in \OF\ot\Zp$, then
\begin{align*}&\abs{D_\cF}_\R^{-1}\bfN_{\cF/\Q}(\beta)^{k-1}\prod_{v|p}W_\beta(\Section,1)|_{s=0}\\
=&
\bfN_{\cF/\Q}(\beta)^{k-1}\prod_{v|p}(1+\nads(\uf_v)+\cdots+\nads(\uf_v)^{v(\beta)})\cdot\abs{D_\cF}_\R^{-1}\abs{D_\cF}_{\Q_p}^{-1}\in\bR.
\end{align*}
 \end{proof}
The following proposition is directly deduced from the construction of the section $\phi_{\ads,s}$ and the description of Hecke action \eqref{E:HeckeAct}. The details are omitted
\begin{prop}\label{P:3.N}Set $\bfD_{\frakC}:=\prod\limits_{v|\frakC^-D_{\cK/\cF}}\cK_v^\x$. Under the assumptions in \propref{P:1.N}, we have
\begin{mylist} \item $\holES$ is an $U_\frakl$-eigenform in $\cM_k(\opcpt_0(\frakl),\bR)$ with the eigenvalue $\ads_+^{-1}(\uf_\frakl)$,
\item  $\holES|[r]=\ads^{-1}(r)\holES$ for $r\in\bfD_\frakC$, \item $\holES(x_n(ta))=\ads^{-1}(a)\holES(x_n(t))$ for $a\in\bfD_\frakC\AF^\x U_n$.\end{mylist}
 \end{prop}

\subsection{}\label{SS:diffOp}
For $\sg\in\bda$ and an integer $n$, let
\[\delta_n^{\sg}:=\frac{1}{2\pii}(\frac{\partial}{\partial
\tau_\sg}+\frac{n}{2iy_\sg})\]
be the
Maass-Shimura's differential operator. Put
\[\delta^{\kappa_\sg\sg}_{k}:=\delta_{k+2\kappa_\sg}^{\sg}\circ\cdots\delta_{k+2}^{\sg}\circ\delta_k^{\sg}\text{ and }
\delta^\kappa_k=\prod_{\sg\in\Sg}\delta_{k}^{\kappa_\sg\sg}.\]  Then the weight raising differential operator $V_{+}^{\kappa_\sg}$ in \secref{S:infinite} is the representation theoretic avatar of $\delta_k^{\kappa_\sg\sg}$ in virtue of the following identity:
\[\delta_k^\kappa=\frac{1}{(-8\pi)^\kappa}V_+^\kappa.\] We thus have
\beq\label{E:41.N}\delta_k^\kappa \bbE^h_{\ads}=\frac{1}{(-4\pi)^{\kappa}}\cdot\frac{\Gamma_\Sg(k\Sg+\kappa)}{\Gamma_\Sg(k\Sg)}\cdot
\bbE^\nh_{\ads}.\eeq


\section{Evaluation of Eisenstein series at CM points}\label{S:EvCM}
\subsection{Period integral}Let $\EucE_\ads:=\wh\holES$ be the \padic avatar of $\holES$ as in \eqref{E:padicavatar.N}. Let $\stt{\Katzd(\sg)}_{\sg\in\Sg}$ be the Dwork-Katz \padic differential operators on \padic modular forms introduced in \cite[Cor.(2.6.25)]{Katz:p_adic_L-function_CM_fields} and let $\Katzd^\kappa=\prod_{\sg\in\Sg}(\Katzd(\sg))^{k_\sg}$. We consider the Hida's measure $\vp^\kappa_\ads:=\vp_{\Katzd^\kappa\EucE_\ads}$ attached to the $U_\frakl$-eigenform $\Katzd^\kappa\EucE_\ads\in V(\opcpt_0(\frakl),\Zbarp)$ as in \eqref{E:measure.N}. Let $\nu$ be a character on $\Cl_n$. We have
\beq\label{E:53.N}\int_{\Cl_\infty}\nu d\vp^\kappa_\ads=\ads_+(\uf_\frakl)^n\cdot\sum_{[t]_n\in\Cl_n}\Katzd^\kappa\EucE_\ads(x_n(t))\pads\nu(t).\eeq
Let $(\Omega_\infty,\Omega_p)\in(\C^\x)^\Sg\x(\Zbarp^\x)^\Sg$ be the complex and \padic CM periods of $(\cK,\Sg)$ introduced in \cite[(4.4 a,b) p.211]{HidaTilouine:KatzPadicL_ASENS} (\cf $(\Omega,c)$ in \cite[(5.1.46), (5.1.48)]{Katz:p_adic_L-function_CM_fields}). From \cite[(2.4.6), (2.6.8), (2.6.33)]{Katz:p_adic_L-function_CM_fields} we can deduce the following important identity:
\beq\label{E:diffCM}\frac{1}{\Omega_p^{k\Sg+2\kappa}}\cdot\Katzd^\kappa\EucE_\ads(x_n(t))=\frac{(2\pii)^{k\Sg+2\kappa}}{\Omega_\infty^{k\Sg+2\kappa}}\cdot \delta^\kappa_k\holES(x_n(t))\quad(t\in(\AKf^{(p)})^\x).\eeq
Therefore, it follows from \eqref{E:41.N} and \eqref{E:53.N} that
\beq\label{E:54.N}\begin{aligned}
\int_{\Cl_\infty}\nu d\vp^\kappa_\ads&=\ads_+(\uf_\frakl)^n\cdot\frac{(2\pii)^{k\Sg+2\kappa}}{\Omega_\infty^{k\Sg+\kappa}}\cdot\sum_{[t]_n\in\Cl_n}\delta^\kappa_k\holES(x_n(t))\ads\nu(t)\\
&=\ads_+(\uf_\frakl)^n\cdot\frac{\pi^{\kappa}\Gamma_\Sg(k\Sg+2\kappa)}{\sqrt{\abs{D_\cF}_\R}\Im(\skewhf)^\kappa\cdot\Omega_\infty^{k\Sg+2\kappa}}\cdot\sum_{[t]_n\in\Cl_n}E^\nh_\ads(\rho(t)\cmpt^{(n)})\ads\nu(t).
\end{aligned}\eeq
Here we choose $t\in(\AKf^{(p\frakl)})^\x$ for a set of representatives $[t]_n\in\Cl_n$ (so $\pads(t)=\ads(t)$).

We shall relate \eqref{E:54.N} to certain period integral of Eisenstein series. First we fix the choices of measures. For each finite place $v$ of $\cF$, let $\dx z_v$ be the normalized Haar measure on $\cK^\x_v$ so that $\vol(R^\x_v,\dx z_v)=1$ and let $\dx t_v=\dx z_v/\dx x_v$ be the quotient measure on $\cK^\x_v/\cF_v^\x$. If $v$ is archimedean, let $\dx t_v$ be the Haar measure on $\cK^\x_v/\cF_v^\x=\C^\x/\R^\x$ normalized so that $\vol(\C^\x/\R^\x,\dx t_v)=1$.
Let $\dx t=\prod'_v\dx t_v$ be a Haar measure on $\AK^\x/\AF^\x$ and let $\dx\bar{t}$ be the quotient measure of $\dx t$ on $\cK^\x\AF^\x\bksl
\AK^\x$ by the counting measure on $\cK^\x$. We define the period integral of $E^\nh_\ads$ by
\[l_\cK(E^\nh_\ads,\nu):=\int_{\cK^\x\AF^\x\bksl
\AK^\x}E^\nh_\ads(\rho(t)\cmpt^{(n)})\ads\nu(t)\dx\bar{t}.\]

Then the last term in \eqref{E:54.N} can be expressed as the following period integral
\[\sum_{[t]_n\in\Cl_n}E^\nh_\ads(\rho(t)\cmpt^{(n)})\ads\nu(t)=\frac{1}{\vol(\ol{U}_n,\dx\bar{t})}l_\cK(E^\nh_\ads,\nu).\]
The rest of this section is devoted to the calculation of $\l_\cK(E^\nh_\ads,\nu)$. For brevity, we write $\phi_v$ for $\Section$ if $v\in\bdh$ and for $\phi^\nh_{k,\kappa_\sg,s,\sg}$ in \eqref{E:5.N} if $v=\sg\in\bda$. The first step is to decompose $l_\cK(E^\nh_\ads,\nu)$ into a product of local integrals $l_{\cK_v}(\phi_v,\nu)$, where
\[l_{\cK_v}(\phi_v,\nu)=\int_{\cK^\x_v/\cF^\x_v}\phi_v(\rho(t)\cmptv^{(n)})\ads\nu(t_v)\dx t_v.\]
Since $\rho:\cF^\x\bksl \cK^\x \to B(\cF)\bksl G(\cF)$ is a bijection, we find that
\begin{align*}
l_\cK(E^\nh_\ads,\nu)&=\int_{\cK^\times \AF^\times\backslash
\AK^\times}E_\A(\rho(t)\cmpt^{(n)},\phi^\nh_{\ads,s})\ads\nu(t)\dx \bar{t}|_{s=0}\\
&=\int_{\cK^\times\AF^\times\bksl\AK^\times}\sum_{\gamma \in G(\cF)/B(\cF)}\phi^\nh_{\ads,s}(\gamma\rho(t)\cmpt^{(n)})\ads\nu(t)\dx \bar{t}\vert_{s=0}\\
&=\int_{\AF^\x\bksl\AK^\x}\phi^\nh_{\ads,s}(\rho(t)\cmpt^{(n)})\ads\nu(t)\dx t\big{\vert}_{s=0}\\
&=\prod_v\int_{\cK^\x_v/\cF^\x_v}\phi_v(\rho(t_v)\cmpt^{(n)}_v)\ads\nu(t_v)\dx t_v|_{s=0}=\prod_v l_{\cK_v}(\phi_v,\nu)|_{s=0}.
\end{align*}

Write $\Kv$ for $\cK_v$ and $\Fv$ for $\cF_v$. In what follows, we suppress $v$ from the notation and proceed to calculate the local integral $l_{\Kv}(\phi_v,\nu)$.
\subsection{$v\in S^\circ$ or $v|\Csplit\Csplit^c$}
In this case, $\phi=f_{\Phi_v,s}$ is the Godement section associated to the \BS function $\Phi_v$ defined in \secref{S:Csplit}. We have
\begin{align*}
l_E(\phi,\nu)=&\int_{\Kv^\times/F^\times}f_\Phi(\rho(t))\adsnu(t)\dx
t\\
=&\int_{\Kv^\times/\Fv^\times}\adsnu(t)|t|_{\Kv}^{s}\dx
t\int_{\Fv^\times}\Phi_v((0,x)\rho(t)\cmpt)\ads_+(x)\abs{x}_{\Fv}^{2s}\dx x\\
=&\int_{\Kv^\times/\Fv^\times}\int_{F^\times}\adsnu(tx)\abs{tx}_{\Kv}^{s}\Phi_v((0,1)\rho(tx)\cmpt)\dx t\dx x\\
=&\int_{\Kv^\times}\adsnu(z)\abs{z}_{\Kv}^{s}\Phi_v((0,1)\rho(z)\cmpt)\dx z\\
=&Z(s,\adsnu,\Phi_E),
\end{align*}
where $\Phi_{\Kv}$ is defined by
\[\Phi_{\Kv}(z):=\Phi_v((0,1)\rho(z)\cmpt).\]
Suppose $v\in S^\circ$. By definition, $\Phi_v=\bbI_{\OFv\oplus\OFv^*}$, and hence $\Phi_{\Kv}=\bbI_{\OKv}$ is the characteristic
function of $\OKv$. It is clear that
\[Z(s,\adsnu,\Phi_\Kv)=L(s,\adsnu).\]

Suppose $v|\Csplit\Csplit$ is a split prime. We write $E=Fe_{\wbar}\oplus Fe_w$ and $\skewhf=-\skewhf_we_{\wbar}+\skewhf_w e_w$ with $w|\Csplit$ as in \secref{S:CM2}. Then by definition,
\[\Phi_{\Kv}(z)=\Phi_v((0,1)\rho(z)\MX{-\skewhf_w}{-\onehalf}{1}{\frac{1}{-2\skewhf_w}})=\vphi_{\wbar}(x)\wh\vphi_w(-\frac{y}{2\skewhf_w})\quad(z=xe_{\wbar}+ye_w).\]
 As $\frakl\ndivide\frakC$, $\nu=(\nu_{\wbar},\nu_{w})$ is unramified at $v$. Therefore,
\begin{align*}
Z(s,\adsnu,\Phi_E)&=\ads_{w}(-2\skewhf_w)\int_{\Kv^\x}\ads_{\wbar}\nu_{\wbar}(x)\ads_{w}\nu_{w}(y)\vphi_{\wbar}(x)\wh\vphi_w(y)\abs{xy}^sd^\times
xd^\times y \\
=&\ads_{w}(-2\skewhf_w)Z(s,\ads_{\wbar}\nu_{\wbar},\vphi_{\wbar})Z(s,\ads_{w}\nu_{w},\wh\vphi_w).\end{align*}
By Tate's local functional equation, we have
\[Z(s,\ads_{w}\nu_{w},\wh\vphi_w)=\frac{L(s,\ads_w\nu_w)\cdot\ads_w\nu_w(-1)}{\ep(s,\ads_w\nu_w,\addchar)L(1-s,\ads_w^{-1}\nu_w^{-1})}\cdot Z(1-s,\ads_w^{-1}\nu_w^{-1},\vphi_w).\]
We find that
\[Z(s,\adsnu,\Phi_E)|_{s=0}=L(0,\ads\nu)\cdot L_\frakS\cdot C_{\Csplit},\]
where
\beq\label{E:splitcons}
L_\frakS=\prod_{\frakq|\frakS}L(1,\ads_\frakq^{-1}\nu_\frakq^{-1})^{-1}\text{ and }C_{\Csplit}=\prod_{w|\Csplit}\frac{\ads_{w}(2\skewhf_w)\cdot\nu_w^{-1}(d_F\uf^{a(\ads_w)})}{\ep(0,\ads_w,\addchar)}.
\eeq
Note that $C_{\Csplit}\in\Zbarp^\x$.
\subsection{$v$ is archimedean or $v|D_{\cK/\cF}\frakC^-$}
Note that $\nu$ is trivial on $\Kv^\x$ if $v$ is inert or ramified because $\nu$ is a character of $\Gal(\cK^-_{\frakl^\infty}/\cK)$. If $v|D_{\cK/\cF}\Cinert$, by the very definition of
$\phi=\Section$ in \eqref{E:Dinert.N}, we find that
\begin{align*}
l_{\Kv}(\phi,\nu)&=\int_{\Kv^\x/\Fv^\x}\phi(\rho(t)\cmptv)\adsnu(t)\dx t\\
&=\vol(\Kv^\x/\Fv^\x,\dx t)\cdot \begin{cases}1&\cdots v|\frakC^-,\\
L(s,\ads_v) &\cdots v|\Cram'.\end{cases}
\end{align*}

If $v=\sg|\infty$ and $\phi=\phi^\nh_{k,\kappa,s,\sg}$ is defined in \eqref{E:5.N}, then
\begin{align*}
l_{\Kv}(\phi,\nu)&=\int_{\Kv^\x/\Fv^\x}\phi^\nh_{k,\kappa,s,\sg}(\rho(t)\MX{\Im\sg(\skewhf)}{}{}{1})\adsnu(t)\dx t\\
&=\vol(\C^\x/\R^\x,\dx t).
\end{align*}

\subsection{$v=\frakl$}A direct computation shows that
\[\rho(t)\cmpt_\frakl^{(n)}=\rho(x+y\bftheta_\frakl)\MX{-b_\frakl d_{\cF_\frakl}\uf_\frakl^n}{1}{a_\frakl d_{\cF_\frakl}\uf_\frakl^n}{0}=\MX{*}{*}{xd_{\cF_\frakl}\uf_\frakl^na_\frakl}{yb_\frakl}\quad(t=x+y\bftheta_\frakl).\]
We thus find that
\[\rho(t)\cmpt_\frakl^{(n)}\in B(F)\bdw N(O_\frakl^*)\iff t\in \cF_\frakl^\x(1+\uf_\frakl^n\OKbasis_\frakl O_\frakl).\]
Let $R_{n,\frakl}=R_n\ot_OO_\frakl$. Then $R_{n,\frakl}^\x=O^\x(1+\uf_\frakl^n\bftheta_\frakl O_\frakl)$. Let $\pi:E^\x\to E^\x/F^\x$ be the quotient map. Thus we have
\begin{align*}l_{\Kv}(\phi,\nu)&=\int_{\Kv^\x/\Fv^\x}\phi(\rho(t)\cmpt_\frakl^{(n)})\adsnu(t)\dx t\\
&=\chi_+(\uf)^{-n}\int_{\Kv^\x/\Fv^\x}\ch_{\pi(R_{n,\frakl}^\x)}(t)\dx t=\chi_+(\uf)^{-n}\vol(\pi(R_{n,\frakl}^\x),\dx t).
\end{align*}

\subsection{Evaluation formula}We summarize our local calculations in the following proposition.
\begin{prop}\label{P:EVCM}Suppose that either of the conditions (1-3) in \propref{P:1.N} holds. Then we have the following evaluation formula:
\[\frac{1}{\Omega_p^{k\Sg+2\kappa}}\cdot\int_{\Cl_\infty}\nu d\vp^\kappa_\ads=\frac{\pi^{\kappa}\Gamma_\Sg(k\Sg+\kappa)L^{(\frakl)}(0,\adsnu)}{\Omega_\infty^{k\Sg+2\kappa}}\cdot
\frac{2^{r}L_\frakS C_{\Csplit}}{\sqrt{\abs{D_\cF}_\R}(\Im \skewhf)^\kappa},\]
where $r$ is the number of prime factors of $D_{\cK/\cF}$.
\end{prop}
\begin{proof}We note that
\begin{align*}
\sum_{[t]_n\in\Cl_n}E_\ads^\nh(\rho(t)\cmpt^{(n)})\ads\nu(t)&=\frac{1}{\vol(\ol{U}_n,\dx\bar{t})}\cdot l_\cK(E_\ads^\nh,\nu)\\
&=\frac{1}{\vol(\ol{U}_n,\dx\bar{t})}\cdot \prod_vl_{\cK_v}(\Section,\nu_v)|_{s=0}\\
&=L^{(\frakl)}(0,\ads\nu)\cdot 2^rL_\frakS C_{\Csplit}\cdot \chi_+(\uf)^{-n}.
\end{align*}
The proposition follows from \eqref{E:54.N} immediately.\end{proof}

\section{Non-vanishing of Eisenstein series modulo $p$}\label{S:NVES}
\subsection{} Throughout this section, we retain the assumptions \eqref{unr}, \eqref{ord} and $(p\frakl,\cD_{\cK/\cF}\frakC)=1$.  Let $\ads$ be a Hecke character of $\cK^\x$ and take $\frakS$ to be a split prime $\frakq$ as in \secref{S:41}. We remark that an auxiliary split prime $\frakq$ is introduced to assure the assumption (2) in \propref{P:1.N}, so the $L$-value in the evaluation formula \propref{P:EVCM} has an extra local factor $L_\frakq=L(1,\ads_\frakq^{-1}\nu_\frakq^{-1})^{-1}$. However, this is harmless to \NV property since the closed subgroup generated by associated Frobenious $\Frob_\frakq$ in $\Gal(\cK^-_\frakl/\cK)$ is non-trivial, and hence the set $\stt{\nu\in\frakX^-_\frakl\mid L(1,\nu_\frakq^{-1}\ads_\frakq^{-1})^{-1}\con 0\pmod{\frakm}}$ is a proper closed subset of $\frakX^-_\frakl$.

To establish \NV property for $(\ads,\frakl)$, by Hida's non-vanishing criterion of a \padic measure associated to eigenforms (\thmref{T:1.N}) and the evaluation formula of our Eisenstein measure $d\vp^\kappa_\ads$ (\propref{P:EVCM}), it suffices to show the non-vanishing modulo $p$ of some Fourier coefficient of $(\Katzd^\kappa\EucE_\ads)^\cR=\Katzd^\kappa\EucE_\ads^\cR$ at some cusp $(\OF,\frakc(\fraka)^{-1})$.
\begin{lm}Put $(\holES)^\cR=\sum_{r\in\cR}\holES|[r]$. Then we have
\[(\holES)^\cR=\#\Delta^\alg \cdot \holES.\]\end{lm}
\begin{proof} It can be shown that $\Delta_\alg$ is generated by ramified primes, so $\cR$ can chosen from elements in $\prod_{v|D_{\cK/\cF}}\cK^\x_v$. 
The lemma follows form \propref{P:3.N} (2).\end{proof}

\begin{Remark}\label{R:1.N}Since $\EucE_\ads^\cR$ is the \padic avatar of $(\holES)^\cR$ and $\#\Delta_\alg$ is a power of $2$, from the above lemma and the following identity (\cf\cite[(1.23)]{HidaTilouine:KatzPadicL_ASENS})
\[\bfa_\beta(\Katzd^\kappa\EucE_\ads,\frakc)=\beta^\kappa \bfa_\beta(\bbE^h_{\ads},\frakc),\]
we conclude that \NV property for $(\ads,\frakl)$ holds if the following hypothesis ($\mathrm{H'}$) is verified:
\begin{description}[\breaklabel\setlabelstyle{\itshape}]\item [$(\mathrm{H'})$]
For every $u\in O_\frakl$ and a positive integer $r$, there exist $\beta\in\OF^\x_\setp$ and $\frakc=\frakc(\fraka)$ such that $\beta\con u\pmod{\frakl^r}$ and \[\bfa_\beta(\bbE^h_{\ads},\frakc)\not\con 0\pmod{\frakm}.\]
\end{description}

\end{Remark}
\subsection{}
Let $\ol{\ads}$ be the reduction modulo $\frakm$ of the \padic avatar of $\ads$ and let $\ol{\ads}_+=\ol{\ads}|_{\AF^\x}$.
Let $\om_\cF:\AF^\x/\cF^\x\to\mu_{p-1}$ be the \Teich character regarded as a Hecke character of $\cF^\x$ via geometrically normalized reciprocity law. We first treat the case $\ads$ is not \emph{residually self-dual}, namely \[\ol{\ads}_+\not\con \qchKF\om_\cF\pmod{\frakm}.\] The following proposition is due to Hida \cite{Hida:nonvanishingmodp}.
\begin{prop}[Hida]\label{P:2.N}Suppose that $\ads$ is not residually self-dual. Then \NV holds for $(\ads,\frakl)$ if for every $v|\frakC^-$ there exists $\Beth_v\in \cF_v^\x$ such that \beq\label{E:NV.N}A_{\Beth_v}(\ads_v)\not\con 0\pmod{\frakm}.\eeq
\end{prop}
\begin{proof} 
We have to verify the hypothesis ($\mathrm{H'}$) in \remref{R:1.N}. Given $u\in\OF_\frakl$ and a positive integer $r$, we extend $\Beth_{\frakC^-}=(\Beth_v)_{v|\frakC^-}$ to an idele $\Beth=(\Beth_v)$ in $\AF^\x$ by taking $\Beth_v=1$ for $v\ndivide \frakl\frakC D_{\cK/\cF}$ or $v|\Csplit\Csplit^c$, $\Beth_\frakl=u$ and $\Beth_v=\uf_v^{-1}$ as in \eqref{E:FC8.N} if $v|\Cram'$.
Let $\frakb^-:=\prod_{\frakq|\frakC^-}\frakq^{M_\frakq}$, $M_\frakq=\max\stt{v_\frakq(\frakC^-),v_\frakq(\frakC^-)-v_\frakq(\eta_v)}$ and put \[U=\stt{(x_\infty,x_f)\in \R_+^{[\cF:\Q]}\x(O\ot_\Z\Zhat)^\x\mid x_f\con 1\pmod{D_{\cK/\cF}\frakl^r\frakb^-}}. \]  Let $\frakc=\frakc(\fraka)$ and $\bfc$ be the associated idele as in \propref{P:1.N} and consider the idele class $[\bfc\Beth^{-1}]:=\cF^\x \bfc\Beth^{-1} U$ in $\AF^\x$. For each idele $a\in \OF\ot_\Z\Zhat$ in the class $[\bfc\Beth^{-1}]$ such that each local component $a_v=1$ at $v|pD_{\cK/\cF}\frakC\frakC^c$, we can write $a=\beta\bfc\eta^{-1}u$ for $\beta\in\OF_\setp^\x$ and $u\in U$, and from the explicit formula of $\bfa_\beta(\holES,\frakc)$ (\propref{P:1.N} combined with \lmref{L:ramified_case.N}, \eqref{E:FC2.N}, \eqref{E:FC3.N} and \eqref{E:FC4.N}), we find that
\beq\label{E:11.N}\begin{aligned}\bfa_\beta(\holES,\frakc)=&C_\beta\cdot \prod_{v\in S^\circ}\frac{1-\nads(\uf_v)^{v(a_v)+1}}{1-\nads(\uf_v)}\cdot\ads_+(\bfc_v),\text{ where}\\
C_\beta=&\abs{D_\cF}^{-1}_\R\cdot \bfN_{\cF/\Q}(\beta)^{k-1}\cdot\prod_{w|\Csplit}\ads_+(\beta)\vphi_{\wbar}(\beta)\abs{\cD_{\cF}^{-1}}_{\cF_v}\cdot\abs{\bfc}_{\cF_\frakl}\cdot
\prod_{w|\Cram'}\ads^{-1}(\bftheta_v)\abs{\uf_v \cD_{\cF}^{-1}}_{\cF_v}\\
&\times\prod_{v|\frakC^-}A_{\eta_v}(\ads_v)\cdot \abs{\cD_{\cF}^{-1}}_{\cF_v}\addchar^\circ(-2^{-1}t_v\beta).\end{aligned}\eeq
By our choices of $\eta$ and $a$, $C_\beta\not\con 0\pmod{\frakm}$.

Suppose that $\bfa_\beta(\holES,\frakc)\con 0\pmod{\frakm}$ for all $\beta\in \OF^\x_\setp$ such that $\beta\con u\pmod{\frakl^r}$. In particular, for every uniformizer $\uf_v\in [\bfc\Beth^{-1}]$ at $v\ndivide p\frakc\frakC\frakC^c D_{\cK/\cF}$, we deduce from \eqref{E:11.N} that $\nads(\uf_v)\con\ol{\ads}_+\om_\cF^{-1}(\uf_v)\con -1\pmod{\frakm}$  Moreover, the argument in \cite[p.780]{Hida:nonvanishingmodp} shows that $\ol{\ads}_+\om_\cF^{-1}$ is a quadratic character of level $U$ and takes value $-1$ on $[\bfc\Beth^{-1}]$. Moving $\frakc=\frakc(\fraka)$ among prime-to-$p\frakC D_{\cK/\cF}$ ideals $\fraka$ of $\OK$, we conclude that $\ol{\ads}_+\om_\cF^{-1}\con\qchKF\pmod{\frakm}$, which is a contradiction.
\end{proof}

\begin{lm}\label{L:2.N} Let $v|\frakC^-$ and $w$ be the place of $\cK$ above $v$. Suppose that $\mu_p(\ads_v)=0$. Then there exists $\Beth_v\in\cF^\x_v$ such that
\[A_{\Beth_v}(\ads_v)\not \con 0\pmod{\frakm}.\]
Moreover, if $v$ is inert and $\ads_v|_{\cF_v^\x}=\tau_{\cK_v/\cF_v}$, then $\Beth_v$ can be further chosen so that $v(\Beth_v)=-w(\frakC^-)$.
\end{lm}
\begin{proof} First we make some observations. Notation is as in \subsecref{S:CramCinert}. We let $F=\cF_v$ and $E=\cK_v$. Let $\uf=\uf_v$ be a uniformizer of $F$ and $\OKbasis=\bftheta_v$. Recall that $\mu_p(\ads_v)=\inf_{x\in E^\x}v_p(\ads_v(x)-1)$, so the assumption $\mu_p(\ads_v)=0$ is equivalent to $\ads|_{E^\x}\not\con 1\pmod{\frakm}$. Since $A_\beta(\ads)=\addchar^\circ(t\beta)\wtd A_\beta(\ads)$, it is equivalent to showing the lemma for $\wtd A_\beta(\ads)$. For an integer $m$ and $a\in F$, we put
\[c_m(a)=\int_{\OFv}\ads^{-1}(a+\uf^m x+\OKbasis)dx.\]
By \eqref{E:ABchi.N}, for $\eta\in \uf^{-m}\OFv^\x$ and every sufficiently large positive integer $M$ (depending on $m$) we have
\[\wtd A_{\eta}(\ads)=\int_{\uf^{-{M}}\OFv}\ads^{-1}(x+\OKbasis)\addchar^\circ(\eta x)dx.\]
Thus for each $a\in F$ we find that\beq\label{E:8.N}
\begin{aligned}\int_{\uf^{-m}\OFv^\x}\wtd A_{\eta}(\ads)\addchar^\circ(-\eta a)d\eta&=\int_{\uf^{-M}\OFv}\ads^{-1}(x+\OKbasis)dx\int_{\uf^{-m}\OFv^\x}\addchar^\circ(\eta(x-a))d\eta\\
&=\int_{\OFv}\ads^{-1}(a+\uf^m x+\OKbasis)dx-\int_{\OFv}\ads^{-1}(a+\uf^{m-1}x+\OKbasis)dx\\
&=c_m(a)-c_{m-1}(a).
\end{aligned}\eeq

Now we prove the first assertion by contradiction. Suppose that $\wtd A_\eta(\ads)\con 0\pmod{\frakm}$ for all $\eta\in F^\x$. The equation \eqref{E:8.N} implies that for every $a\in F$, $c_m(a)\pmod{\frakm}$ is a constant independent of $m$. Taking a sufficiently large $m_0$, we find that $c_m(a)\con c_{m_0}(a)=\ads^{-1}(a+\OKbasis)\pmod{\frakm}$ for every integer $m$ and $a\in F$. On the other hand, it is clear that $c_m(a)=c_m(a')$ whenever $a,a'\in \uf^m\OFv$, so we conclude that the function $a\mapsto \ads^{-1}(a+\OKbasis)\pmod{\frakm}$ is the constant function $\ads^{-1}(\OKbasis)$ on $F$, and hence $\ads(1+a\OKbasis)\con 1\pmod{\frakm}$ for all $a\in F$. This implies that $\ads_v\con 1\pmod{\frakm}$, which is a contradiction.

We proceed to prove the second assertion. Suppose that $v$ is inert and $\ads_v|_{F}=\tau_{E/F}$. Note that in this case $\mu_p(\ads_v)=0$ is equivalent to $\ads_v|_{\cO^\x_{E}}\not \con 1\pmod{\frakm}$. Let $m=w(\frakC^-)\geq 1$. If $\wtd A_{\eta}(\ads)\con 0\pmod{\frakm}$ for all $\eta\in\uf^{-m}\OFv^\x$, then it follows from \eqref{E:8.N} that
\[\ads^{-1}(a+\OKbasis)=c_m(a)\con c_{m-1}(0)\pmod{\frakm}\text{ for }a\in\uf^{m-1}\OFv.\] Therefore, $a\mapsto \ads^{-1}(a+\OKbasis)\pmod{\frakm}$ is the constant function $\ads^{-1}(\OKbasis)$ on $\uf^{m-1}\OFv$, and hence $\ads(1+a\OKbasis)\con 1\pmod{\frakm}$ for all $a\in\uf^{m-1}\OFv$. If $m=w(\frakC^-)>1$, this is impossible, and if $m=1$, this contradicts to the assumption that $\ads|_{\cO^\x_{E}}\not\con 1\pmod{\frakm}$.
\end{proof}

 The following corollary is an immediate consequence of \propref{P:2.N} and \lmref{L:2.N}, which gives a partial generalization of Hida's theorem.
\begin{cor}\label{C:1.N}  Suppose that the following conditions hold:
\begin{itemize}
\item[(L)] $\mu_p(\ads_v)=0$ for every $v|\frakC^-,$
\item[(N)] $\ads$ is not residually self-dual.
\end{itemize}
Then \NV holds for $(\ads,\frakl)$.
\end{cor}

\subsection{}
\def\kap{\nads}
We consider the self-dual case.  First we recall the following lemma on local root numbers of self-dual characters.
\begin{lm}[Prop. 3.7 \cite{Murase-Sugano:Local_theory_primitive_theta}]\label{L:MSROOT.N}Let $\ads$ be a self-dual character, \ie $\ads|_{\AF}=\qchKF\Abs_{\AF}$. Then
\begin{mylist}\item $W(\kap_v)=\pm \kap_v(2\skewhf)$.
\item If $v$ is split, $W(\kap_v)=\kap_v(2\skewhf)$.
\item If $v$ is inert,
$W(\kap_v)=(-1)^{a(\kap_v)+v(\frakc(\OK))}\kap_v(2\skewhf)$ $(\frakc(\OK)=\cD_\cF^{-1}(2\skewhf\cD_{\cK/\cF}^{-1}))$.
\end{mylist}
\end{lm}
\begin{prop}\label{P:main.N}Let $\ads$ be a self-dual character of the global root number $W(\kap)=+1$ $(\nads=\ads\Abs_{\AK}^{-\onehalf})$. Suppose that $\frakl$ splits in $\cK$ and that
there exists $\Beth_v\in \cF^\x_v$ for each $v|\frakC^-$ such that
\begin{itemize}
\item[(i)] $A_{\Beth_v}(\ads_v)\not\con 0\pmod{\frakm}$,
\item[(ii)] $W(\kap_v)\qchKF(\Beth_v)=\kap_v(2\skewhf)$.
\end{itemize}
Then \NV holds for $(\ads,\frakl)$.
\end{prop}
\begin{proof} We need to verify the hypothesis ($\mathrm{H'}$) in \remref{R:1.N}. Given $u\in O_\frakl$ and a positive integer $r$, we extend $(\Beth_v)_{v|\frakC^-}$ to an idele $\Beth=(\Beth_v)$ in $\AF^\x$ such that
\begin{itemize}\item $\Beth_\frakl\con u\mod{\frakl^r}$ and $\Beth_v=1$ for every split prime $v\not
=\frakl$,
\item $W(\kap_v)\qchKF(\Beth_v)=\kap_v(2\skewhf)$ for every $v|\bdh$.
\end{itemize}
By \lmref{L:MSROOT.N}, this is possible since $\frakl$ splits in $\cK$.
On the other hand, it is well known that $W(\kap_\sg)=i^{2\kappa_\sg+1}=\kap_\sg(\skewhf)$ since $\kap_\sg(z)=\frac{z}{\abs{z}}\cdot(\frac{z}{\zbar})^{\kappa_\sg}$ for $\sg\in\Sg$ (\cf \cite[p.13]{Tate:Number_theoretic_bk}). From the assumption on the global root number $W(\kap)=\prod_v
W(\kap_v)=1$, we deduce that
\[\prod_{v\in\bdh}W(\kap_v)=\prod_{v\in\bdh}\kap_v(2\skewhf).\]
This implies that $\qchKF(\Beth)=1$, so we can write
\[\Beth=\beta\norm{a},\,\beta\in\cF_+,a\in\A_\cK^\x.\]Moreover by the approximation theorem, the idele $a$ can be further chosen so that $a\con
1\mod{p\el^r\frakC^N}$ for any sufficiently large $N$. Note that
\[W(\kap_v)\tau_{\cK_w/\cF_v}(\beta)=W(\kap_v)\qchKF(\Beth_v)=\kap_v(2\skewhf).\]
For every sufficiently small $\ep$, we have thus constructed $\beta\in\cF_+\cap\OF^\x_{(p\Csplit\Csplit^c)}$
such that
\begin{itemize}
\item $\beta\con u\pmod{\frakl^r},$
\item $\abs{\beta-\Beth_v}_{\cF_v}<\ep$ for all $v|p\frakC\frakC^c,$
\item $W(\kap_v)\kap_v(\beta)=\kap_v(2\skewhf)$ for all $v\in \bdh$.
\end{itemize}
Here we let $\ep$ be sufficiently small so that $A_{\beta}(\ads_v)=A_{\Beth_v}(\ads_v)$ for $v|\frakC^-$.
Recall that $v(\frakc(\OK))=0$ for $v|p\frakC\frakC^c$ by our choice of $\skewhf$. By \lmref{L:MSROOT.N} (3), we find that $v(\beta)\con v(\frakc(\OK))\pmod{2}$ for every inert place $v\ndivide\frakC^-$. It follows that there exists a fractional $\fraka$ of $\OK$ such that
\[\prod_{\frakq|\frakC^-}\frakq^{v_\frakq(\beta)}=(\beta)\frakc(\OK)\norm{\fraka}^{-1}=(\beta)\frakc(\fraka).\]
Define $\bfc\in\AFf^\x$ as follows: $\bfc_v=\beta^{-1}$ if $v\ndivide p\frakl\frakC\frakC^c$, $\bfc_v=1$ if $v|p\frakC\frakC^c$. Then $\il_\cF(\bfc)=\frakc(\fraka)$ by the choice of $\beta$ and $\frakc(\fraka)$. From \propref{P:1.N}, \eqref{E:FC2.N}, \eqref{E:FC3.N}, \eqref{E:FC4.N} and \eqref{E:10.N}, we find that the $\beta$-th Fourier coefficient $\bfa_\beta(\holES,\frakc)$ of $\bbE^h_\ads$ at the cusp $(\OF,\frakc^{-1})$ is given by
\begin{align*}\bfa_\beta(\holES,\frakc)=&\frac{1}{\abs{D_\cF}_\R}\cdot\prod_{v\in\bdh}W_\beta(\Section,\MX{1}{}{}{\bfc^{-1}_v})|_{s=0}\\
=&\ads_+(\bfc)\prod_{w|\Csplit}\ads_w(\beta)\cdot \prod_{v|\frakC^-}A_\beta(\ads_v)\cdot\abs{\cD_{\cF}^{-1}}_{\cF_v}\addchar^\circ(-2^{-1}t_v\beta).
\end{align*}
It is clear that the non-vanishing of $\bfa_\beta(\holES,\frakc)\pmod{\frakm}$ is equivalent to
\[A_{\beta}(\ads_v)=A_{\Beth_v}(\ads_v)\not\con 0\pmod{\frakm}\text{ for every }v|\frakC^-.\qedhere\]\end{proof}

Now we are ready to prove our main result.
\begin{thm}\label{T:main.N} Suppose that $\frakl$ splits in $\cK$. Let $\ads$ be a self-dual Hecke character such that \begin{mylist}
\item[\rm{(L)}] $\mu_p(\ads_v)=0$ for every $v|\frakC^-,$
\item[\rm{(R)}] The global root number $W(\nads)=1,$
\item[\rm{(C)}] $\Cram$ is square-free.
\end{mylist}
Then \NV holds for $(\ads,\frakl)$.
\end{thm}
\begin{proof}
It suffices to verify that for each $v|\frakC^-$ there exists $\eta_v\in\cF_v^\x$ which satisfies (i) and (ii) in \propref{P:main.N}. For $v|\Cram$, we take $\Beth_v\in\cF_v^\x$ such that $W(\nads_v)=\qchKF(\Beth_v)\nads_v(2\skewhf)$. Note that the assumption (C) implies that $v\ndivides 2$. By \propref{P:formulaRamified.N} (3)
we find that
\begin{align*}A_{\Beth_v}(\ads)=&(\nads_v(-2\UF_v^{-1}d_{\cF_v})+\nads_v(2^{-1}\Beth_v)W(\nads_v))\cdot\ads_v(-2^{-1}d_F^{-1})\abs{\uf}^\onehalf\\
=&(\nads_v(\skewhf)+\nads(\skewhf))\ads_v(-2^{-1}d_F^{-1})\abs{\uf}^\onehalf\quad(2\skewhf=d_{\cF_v}\UF_v)\\
=&2\nads_v(\skewhf)\ads_v(-2^{-1}d_F^{-1})\abs{\uf}^\onehalf\not\con 0\pmod{\frakm}.\end{align*}
For $v|\Cinert$, we choose $\Beth_v$ to be as in \lmref{L:2.N}, so $A_{\Beth_v}(\ads_v)\not\con 0\pmod{\frakm}$ and \[v(\Beth_v)=w(\frakC^-)=a(\nads_v)+v(\frakc(\OK)).\] It follows from \lmref{L:MSROOT.N} (3) that $W(\nads_v)=\qchKF(\Beth_v)\nads_v(2\skewhf)$.
\end{proof}

\begin{Remark}We give a few remarks on \thmref{T:main.N}:\begin{enumerate}
\item The assumption (C) has been removed in view of \cite[Prop. 6.3]{Hsieh:VMU}.
\item Let $\ads_1$ be a self-dual character and $\nu$ be a finite order character such that $\nu$ has prime-to-$p$ conductor and $\nu\con 1\pmod{\frakm}$. As pointed out by the referee, one can prove \thmref{T:main.N} for $\ads:=\ads_1\nu$, keeping (L) and (C) but replacing (R) by the condition (Rm): $W(\ads^*)\con 1\pmod{\frakm}$, which implies the condition (R) for $\ads_1$. Indeed, as $\nu$ must have square-free conductor, (C) holds for $\ads_1$, and (L) obviously holds for $\ads_1$ as well. Thus $\ads_1$ satisfies the hypothese in \thmref{T:main.N}, and for every $u\in\OF_\frakl$ and $r$, we can choose $\beta\in\cF_+$ as in the proof of \propref{P:main.N} such that $\bfa_\beta(\bbE^h_{\ads_1},\frakc)\not\con 0\pmod{\frakm}$. By the condition (L) the supports of the conductors of $\ads$ and $\ads_1$ only differ by split primes, we find that $\bfa_\beta(\holES,\frakc)\not\con 0\pmod{\frakm}$.
\item From the expression of $A_\beta(\ads)$ in \eqref{E:ABchi.N}, it is not difficult to deduce that if $\mu_p(\ads_v)>0$ for some place $v|\frakC^-$, then $A_\beta(\ads_v)\con 0\pmod{\frakm}$ for all $\beta\in\cF_v^\x$, and hence $\holES\con 0\pmod{\frakm}$ by $q$-expansion principle (\cf \cite[Prop.6.2]{Hsieh:VMU} for the self-dual case). In addition, in the self-dual case, we can deduce from \cite[Lemma 6.1]{Hsieh:VMU} that if $W(\nads)=-1$, then $\bfa_\beta(\holES,\frakc)=0$ for all $\beta$. It follows that $\holES|_{\frakc}=0$.
\item The assumption that $\frakl$ splits in $\cK$ is only used in \propref{P:main.N} to assure the local root number $W(\nads_\frakl)$ for all $u\in\OF_\frakl$ and $r$ satisfies certain epsilon dichotomy $W(\nads_\frakl)\tau_{\cK_\frakl/\cF_\frakl}(\eta_v)=\nads_\frakl(2\skewhf)$ whenever $\eta_v\con u\pmod{\frakl^r}$. This is false for nonsplit $\frakl$. For example, when $\frakl$ is inert, this dichotomy holds precisely when $v_\frakl(\eta_v)\con v_\frakl(\frakc(\OK))\pmod{2}$. To treat nonsplit $\frakl$, it seems that one has to refine Theorem 3.2 in \cite{Hida:nonvanishingmodp} (at least when $\frakl$ has degree one over $\Q$).

\end{enumerate}
\end{Remark}

\bibliographystyle{amsalpha}
\bibliography{C:/users/MingLun/texsetting/mybib}

\providecommand{\bysame}{\leavevmode\hbox to3em{\hrulefill}\thinspace}
\providecommand{\MR}{\relax\ifhmode\unskip\space\fi MR }
\providecommand{\MRhref}[2]{%
  \href{http://www.ams.org/mathscinet-getitem?mr=#1}{#2}
}
\providecommand{\href}[2]{#2}
\begin{thebibliography}{Hid04b}

\bibitem[Fin06]{Finis:nonvanishingell}
T.~Finis, \emph{Divisibility of anticyclotomic {$L$}-functions and theta
  functions with complex multiplication}, Ann. of Math. (2) \textbf{163}
  (2006), no.~3, 767--807.

\bibitem[Hid04a]{Hida:nonvanishingmodp}
H.~Hida, \emph{Non-vanishing modulo {$p$} of {H}ecke {$L$}-values}, Geometric
  aspects of Dwork theory. Vol. I, II, Walter de Gruyter GmbH \& Co. KG,
  Berlin, 2004, pp.~735--784.

\bibitem[Hid04b]{Hida:p-adic-automorphic-forms}
\bysame, \emph{{$p$}-adic automorphic forms on {S}himura varieties}, Springer
  Monographs in Mathematics, Springer-Verlag, New York, 2004.

\bibitem[Hid07]{Hida:nonvanishingnew}
\bysame, \emph{Non-vanishing modulo {$p$} of {H}ecke {$L$}-values and
  application}, {$L$}-functions and {G}alois representations, London Math. Soc.
  Lecture Note Ser., vol. 320, Cambridge Univ. Press, Cambridge, 2007,
  pp.~207--269.

\bibitem[Hid10]{Hida:mu_invariant}
\bysame, \emph{The {I}wasawa {$\mu$}-invariant of {$p$}-adic {H}ecke
  {$L$}-functions}, Ann. of Math. (2) \textbf{172} (2010), no.~1, 41--137.

\bibitem[Hsi11]{Hsieh:ESU21}
M.-L. Hsieh, \emph{Eisenstein congruence on unitary groups and {I}wasawa main
  conjecture for {CM} fields}, Preprint is available at
  "http://www.math.ntu.edu.tw/~mlhsieh/research.htm", 2011.

\bibitem[Hsi12]{Hsieh:VMU}
\bysame, \emph{On the $\mu$-invariant of anticyclotomic $p$-adic
  {$L$}-functions for {CM} fields}, to appear in J. Reine Angew. Math.
  DOI:10.1515/crelle-2012-0056, 2012.

\bibitem[HT93]{HidaTilouine:KatzPadicL_ASENS}
H.~Hida and J.~Tilouine, \emph{Anti-cyclotomic {K}atz {$p$}-adic
  {$L$}-functions and congruence modules}, Ann. Sci. \'Ecole Norm. Sup. (4)
  \textbf{26} (1993), no.~2, 189--259.

\bibitem[JL70]{Jacquet_Langlands:GLtwo}
H.~Jacquet and R.~P. Langlands, \emph{Automorphic forms on {${\rm GL}(2)$}},
  Lecture Notes in Mathematics, Vol. 114, Springer-Verlag, Berlin, 1970.

\bibitem[Kat78]{Katz:p_adic_L-function_CM_fields}
N.~Katz, \emph{{$p$}-adic {$L$}-functions for {CM} fields}, Invent. Math.
  \textbf{49} (1978), no.~3, 199--297.

\bibitem[Kot92]{Kottwitz:Points-On-Shimura-Varieties}
R.~Kottwitz, \emph{Points on some {S}himura varieties over finite fields},
  Journal of AMS \textbf{5} (1992), no.~2, 373--443.

\bibitem[MS00]{Murase-Sugano:Local_theory_primitive_theta}
A.~Murase and T.~Sugano, \emph{Local theory of primitive theta functions},
  Compositio Math. \textbf{123} (2000), no.~3, 273--302.

\bibitem[Roh82]{Rohrlich:root_number}
D.~Rohrlich, \emph{Root numbers of {H}ecke {$L$}-functions of {CM} fields},
  Amer. J. Math. \textbf{104} (1982), no.~3, 517--543.

\bibitem[SGA64]{SGA3-2}
\emph{Sch\'emas en groupes. {II}: {G}roupes de type multiplicatif, et structure
  des sch\'emas en groupes g\'en\'eraux}, S\'eminaire de G\'eom\'etrie
  Alg\'ebrique du Bois Marie 1962/64 (SGA 3). Dirig\'e par M. Demazure et A.
  Grothendieck. Lecture Notes in Mathematics, Vol. 152, Springer-Verlag,
  Berlin, 1962/1964.

\bibitem[Shi98]{Shimura:ABV-with-CM}
G.~Shimura, \emph{Abelian varieties with complex multiplication and modular
  functions}, Princeton Mathematical Series, vol.~46, Princeton University
  Press, Princeton, NJ, 1998.

\bibitem[Tat79]{Tate:Number_theoretic_bk}
J.~Tate, \emph{Number theoretic background}, Automorphic forms, representations
  and {$L$}-functions ({P}roc. {S}ympos. {P}ure {M}ath., {O}regon {S}tate
  {U}niv., {C}orvallis, {O}re., 1977), {P}art 2, Proc. Sympos. Pure Math.,
  XXXIII, Amer. Math. Soc., Providence, R.I., 1979, pp.~3--26.

\end{thebibliography}
\end{document}